\documentclass[12pt,twoside]{article}
\usepackage{amsmath,amssymb,mathrsfs,graphicx}
\usepackage[small,sc]{caption} 
\usepackage{xcolor}
\usepackage{tikz} 
\usepackage[normalem]{ulem}
\usetikzlibrary{decorations.pathreplacing}
\usepackage{comment}

\newtheorem{theorem}{Theorem}[section]
\newtheorem{lemma}[theorem]{Lemma}

\newtheorem{corollary}[theorem]{Corollary}
\newtheorem{definition}[theorem]{Definition}

\newtheorem{rmrk}[theorem]{Remark}

\DeclareMathAlphabet{\mathbfit}{OML}{cmm}{b}{it}

\makeatletter
\@addtoreset{equation}{section}
\makeatother

\newenvironment{remark}
{\begin{rmrk} \em}
{\end{rmrk}}

\newcommand{\Z} {\mathbb{Z}}

\newcommand{\qed} {\hfill {\small Q.E.D.} \par\medskip}

\newcommand{\proof} {\noindent \textsc{Proof.} }
\newcommand{\proofof}[1] {\noindent \textsc{Proof of {#1}.} }

\newcommand{\cM}{\mathcal{M}}
\newcommand{\cF}{\mathcal{F}}

\newcommand{\cR}{\mathcal{R}}
\newcommand{\cK}{\mathcal{K}}

\newcommand{\bN}{\mathbb{N}}
\newcommand{\bR}{\mathbb R}
\newcommand{\cJ}{\mathcal J}
\newcommand{\fA}{\mathfrak A}
\newcommand{\fB}{\mathfrak B}
\newcommand{\fT}{\mathfrak T}
\newcommand{\cT}{\mathcal T}
\newcommand{\cI}{\mathcal I}
\newcommand{\fI}{\mathfrak I}
\newcommand{\fS}{\mathfrak S}
\newcommand{\fF}{\mathfrak F}
\newcommand{\cS}{\mathcal S}
\newcommand{\cH}{\mathcal H}
\newcommand{\ve}{\varepsilon}

\newcommand{\cA}{\mathcal A}
\newcommand{\cC}{\mathcal C}
\newcommand{\cD}{\mathcal D}
\newcommand{\cE}{\mathcal E}
\newcommand{\cG}{\mathcal G}
\newcommand{\cO}{\mathcal O}
\newcommand{\fO}{\mathfrak O}

\newcommand{\bH}{\mathbb H}
\newcommand{\vf}{\varphi}

\renewcommand{\emptyset} {\varnothing}

\newcommand{\cantor}{\mathfrak R}

\newcommand{\brush} {\mathbb B}

\newcommand{\Erg}{\text{Erg}}

\begin{document}

\title{\textbf{K-mixing for aperiodic Lorentz gases}}

\author{\scshape
Giovanni Canestrari\thanks{
Department of Mathematics, University of Toronto, 
40 St.~George Street, Toronto, ON M5S 2E4, Canada. 
Email: {\texttt{giovanni.canestrari@utoronto.ca}}.
}
\ and
Marco Lenci\thanks{
Dipartimento di Fisica e Astronomia, Universit\`a di Bologna, 
Via Irnerio 46, 40126 Bologna, Italy. 
E-mail: {\texttt{marco.lenci@unibo.it}}.}
\thanks{
Istituto Nazionale di Fisica Nucleare,
Sezione di Bologna, Viale Berti Pichat 6/2,
40127 Bologna, Italy.}
}

\maketitle

\begin{abstract}
We prove a general theorem on the decomposition into K-mixing components for recurrent, piecewise-smooth hyperbolic maps in two dimensions. In contrast with classical results, we do not assume that the map preserves a probability measure. As an application, we show that every recurrent aperiodic Lorentz gas satisfying certain uniformity assumptions is K-mixing. Finally, we discuss some consequences of K-mixing for the decay of correlations relative to zero-mean $L^1$ observables in aperiodic Lorentz gases.

\bigskip\noindent 
\textbf{Mathematics Subject Classification (2020):} 37A40, 37A25, 37E05, 37D25, 37C25. 
  
\bigskip\noindent
\textbf{Keywords:} hyperbolic billiards, K-automorphism, K-property, infinite ergodic theory, smooth ergodic theory, Lorentz tubes, chaos.
  
\end{abstract}

\section{Introduction}
\label{sec-intro}

Consider the measure space \((\cM, \fA, \mu)\) and a measurable invertible transformation \(\cF : \cM \to \cM\). The transformation \(\cF\) is said to be \emph{K-mixing} (or to have the \emph{K-property}, or to be a \emph{K-automorphism}, or simply to be \emph{K}) if there exists a sub-$\sigma$-algebra \(\fS \subset \fA\) such that
\begin{equation}\label{def:K}
\begin{split}
 \mathfrak{S} \subseteq \cF \mathfrak{S}, \qquad \bigvee_{n = 0}^{\infty} \cF^{n}\mathfrak{S} = \mathfrak{A}, \qquad \bigcap_{n = 0}^{\infty} \cF^{-n}\mathfrak{S} =\mathfrak N.
\end{split}
\end{equation}
Here, \(\cF \fS = \{\cF(B) \,:\, B \in \fS\}\), \(\vee_{n=0}^{\infty} \cF^n \fS\) is the intersection of all complete sub-$\sigma$-algebras each of which contains all \(\cF^n \fS\) and \(\mathfrak N\) is the trivial $\sigma$-algebra. The conditions \eqref{def:K} can be suggestively described by saying that the future is asymptotically independent of the past.%
\footnote{See the introduction of \cite{MR4514477} for a slightly more technical explanation of this point.}
In other words, the dynamics causes a strong form of memory loss of the initial conditions of the system. In particular, in the case where \(\mu\) is a probability measure, K-mixing implies mixing of all orders. 

After its introduction by Kolmogorov \cite{MR103254}, the K-property became one of the principal benchmarks for strong stochastic behavior. The situation changed somewhat in the 1970s when, due to the work of Ornstein and others, the Bernoulli property replaced the K-property as the ultimate \emph{non-quantitative} ergodic property (a quite incomplete list of references, especially relevant for this paper, includes \cite{MR257322,MR274718,MR355003,MR1375125}). It is hard to argue against this development, as Bernoulliness means by definition an isomorphism with a Bernoulli shift, the ``most random'' stochastic process. But this notion is intrinsically probabilistic and makes no sense in the context of infinite ergodic theory \cite{MR1450400, MR4647080}, within which the present paper is situated.

We consider two-dimensional hyperbolic dynamical systems which are Poincaré recurrent w.r.t.\ a natural, possibly infinite, measure. Under abstract assumptions inspired by the theory of hyperbolic billiards, we prove that these systems admit a \emph{K-decomposition}, that is, a decomposition into countably many cycles whose associated return maps are K. As an application of our abstract theory, we prove that a large class of planar aperiodic Lorentz gases (ALGs) are K-mixing.

The ALGs are billiards models for which the configuration of scatterers is not periodic, as opposed to periodic Lorentz gases (PLGs). Both PLGs and ALGs are natural models in statistical physics and were first employed to describe the motion of electrons in a crystal \cite{Lo1905}. In this context, aperiodicity is particularly relevant as it models the presence of impurities. Moreover, as a bare-bone model of a gas, an ALG is clearly more realistic than a PLG.%
\footnote{See footnote 1 of \cite{MR2842891} for a brief history of the Lorentz model in the context of the kinetic theory of gases.} Recently, PLGs and ALGs gathered even more attention from the ergodic theory community because of their role in infinite ergodic theory (cf.\ \cite{MR4897334} and references therein). 

Without aiming at a precise classification, hyperbolic dynamical systems in infinite measure can be divided into two main categories: those that are non-uniformly hyperbolic (which often model intermittent behavior; see, e.g., \cite{MR576270, MR1695915}) and those that are uniformly hyperbolic, but their natural phase space is non-compact. The latter category can be further divided into two classes: those that possess some sort of periodic structure (for example, \(\mathbb Z\)-extensions of an Anosov map) and those that are aperiodic. We remark that even for the first class, understanding the dynamics on the extended phase space is much more challenging than understanding the base system. Contrary to the periodic case (see, e.g., \cite{MR4514477, MR4752568, MR4396670, MR4970176, MR4866346, CastorriniRavotti2024} and references therein), the literature addressing truly aperiodic systems is rather skewed. Specifically, very little is known about ALGs. In \cite{MR1992667} it was shown that recurrent ALGs are ergodic, and in \cite{MR2237472, MR2720043, MR2842891} that there exist plenty of recurrent ALGs (the set of recurrent ALGs is typical in a certain topological sense). Finally, quenched ergodicity was shown for classes of random Lorentz tubes \cite{MR2741899, MR2820038}. 

As for mixing properties, \cite{MR4507923} shows K-mixing for certain compact, finite-measure systems which describe the billiard dynamics of Lorentz gases defined by aperiodic repetitive local configurations of scatterers of finite complexity. In \cite{MR4604909} the authors prove exponential decay of correlations for sequential billiards. This allows them to compare the sequence of visited scatterers in a class of ALGs with ``lazy gates'' to a stochastic process. Finally, in \cite{MR4505391} the property of global-local mixing (see below) was investigated for certain mechanical systems preserving an infinite measure, including a class of asymptotically periodic Lorentz gases.

Achieving a mixing result for truly aperiodic systems is the main contribution of the present paper. Our arguments are inspired by, and makes use of, tools from the theory of Sinai billiards. We emphasize, however, that the classical argument for K-mixing in the finite-measure case (see, e.g., \cite{MR467837}) is not applicable in our setting because it uses entropy, a concept that does not easily extend to infinite-measure contexts. Therefore, in order to prove K-mixing, we apply a simple criterion, originally developed in \cite{MR3009111} for the exactness of non-invertible dynamical systems. Our method gives also an alternative to Sinai's classical proof of the K-property for hyperbolic billiards \cite{MR274721, MR197684}.

Our main technical assumption is the existence of \emph{Cantor rectangles} throughout the phase space (see (H4) in Section \ref{sec-intro}). Cantor rectangles generate Markov partitions for discontinuous hyperbolic systems and have found their main application in the proof of decay of correlations (see, e.g., the classical papers \cite{MR597749, MR606459} on Sinai billiards). Because of their undeniable usefulness, Cantor rectangles have been constructed in a variety of settings (limiting ourselves to billiards, see \cite{MR1071936} for infinite-horizon billiards and billiards with corners and \cite{MR2041267, MR2150341} for non-uniformly hyperbolic billiards). We refer to \cite[Section 7.3]{MR2229799} for a more comprehensive historical review of the topic.

A rough idea of the proof is the following. Any positive-measure set \(A\) in the tail of \(\fS\) has a density point w.r.t.\ the Lebesgue measure on some unstable manifold contained in a Cantor rectangle \(\cantor\). By recurrence, this density point returns to \(\cantor\) infinitely often. Exploiting the product structure of the rectangle and the expansion on unstable manifolds, this shows that an iterate of \(A\) “fills" the Cantor rectangle almost completely. Since the orbit of \(\cantor\) recurs to itself infinitely often, the orbit of any set \(A\) as above inherits the same recurrence properties as \(\cantor\). This heuristic picture (together with additional arguments) implies that the tail \(\sigma\)-algebra of \(\fS\) is atomic.      

It would be interesting to investigate whether the abstract theory developed in this paper, or some adaptation thereof, can also be applied to other infinite-measure hyperbolic systems, such as those investigated in \cite{MR1930574, MR2505302, MR5005377}. We do not pursue this goal here.

K-mixing also has the following implication for the ``decay of correlation'' in infinite-measure systems. Let \(\cG\) be the closure in \(L^{\infty}\) of \(\cup_{n=0}^{\infty}\cF^n \fS\). For any \(g,h \in L^1(\mu)\), \(\int g  = \int h \), and \(F\in \cG\), we have
\begin{equation}\label{eq:coalescence}
\lim_{n \to \infty} \biggl(\,\int_{\cM} (F\circ \cF^n) g \, d\mu- \int_{\cM} (F\circ \cF^n) h \, d\mu \,\biggr)= 0
\end{equation}
(cf.\ \cite[Theorem 3.5]{MR3235352} and Corollary \ref{cor:decay-K-mix} below). This implies that, if one has two different distributions of the initial condition, the two respective random variables \(F\circ \cF^n\) become asymptotically identical. Notice that, by the first two items of \eqref{def:K}, the sequence \(\{\cF^k \fS\}_k\) approximates the original $\sigma$-algebra \(\mathfrak A\) arbitrarily well, so the assumption on the observable \(F\) is mild. A somewhat stronger (more uniform) statement is obtained for ALGs in Theorem \ref{thm:decay-of-correlations} below.

To conclude, we mention that K-mixing is closely related to other definitions of mixing for infinite-measure systems. In particular, K-mixing (or exactness, its counterpart for non-invertible transformations) has been used as an intermediate step in proving \emph{global-local mixing} for a variety of systems (cf.\ \cite{MR3867232, MR4201832, MR3639443, CanestrariLenci2024, CoatesMelbourne2025}). Global-local mixing is one of the notions of mixing that have been devised in the context of infinite ergodic theory (we refer to \cite{MR2669446} for a detailed discussion of this and related notions%
\footnote{For example, \eqref{eq:coalescence} is called (M3) in \cite{MR2669446} and renamed global-local mixing of type 1, or (GLM1), in \cite{MR3235352}.}%
). 

The paper is structured as follows: in Section \ref{sec:settingandresults} we state the main abstract result for the K-decomposition, whose proof is postponed to Section \ref{sec:proof}. In Section \ref{sec:aperiodic-LG} we apply the abstract theorem to billiards and we prove that a class of aperiodic Lorentz gases are K. Also in Section \ref{sec:aperiodic-LG}, we state and prove Theorem \ref{thm:decay-of-correlations} on correlations. 
\medskip

Throughout the paper \(\bN = \{1,2,...\}\), \(\bN_0 = \bN\cup\{0\}\), \(\cF : \cM \to \cM\) and the phase space \(\cM\) is a Riemannian manifold. For any two points \(p,q \in \cM\) and \(r>0\), we denote by \(|p-q|\) the distance induced by the metric and by \(B(p,r)\) the open ball centered at \(p\) of radius \(r\) w.r.t.\ this metric. We denote by \(\mu\) a measure equivalent to the natural area form. Unless otherwise specified, we use the terminology `almost everywhere' (or mod \(0\), or a.e.\ point. etc.)\ referred to the measure \(\mu\). For a generic $\sigma$-algebra \(\fB\), we denote by \(\fB_+\) the collection of sets \(P \in \fB\) of strictly positive measure. We also denote by \(L^{\infty}(\fB)\) the \(\fB\)-measurable essentially bounded functions. For any two sets \(A\) and \(B\), \(A\triangle B = (A\setminus B) \cup (B \setminus A)\). Finally, given a \(\cC^1\) curve \(W\subset \cM\), \(|W|\) indicates the length of \(W\).

\paragraph*{Acknowledgments.} 
This research was partially supported by the PRIN Grants 2017S35EHN and 2022NTKXCX
of the Ministry of University and Research (MUR), Italy. It is also part of the authors' activity within the DinAmicI community \linebreak (\texttt{www.dinamici.org}) and within the Gruppo Nazionale di Fisica Matematica, INdAM, Italy. C.~G.~was also partially supported by the MUR Excellence Department Projects MatMod@TOV and Math@TOV, awarded to the Department of Mathematics, Universit\`a di Roma Tor Vergata. Part of this article was written during C.~G.'s post-doctoral fellowship at the University of Bologna in 2024-25.

\section{Setting and results}
 \label{sec:settingandresults}

Let \(\cM\) be a two-dimensional Riemannian manifold (perhaps with boundary and corners), \(\fA\) the completion of the Borel $\sigma$-algebra and \(\mu\) be a locally finite measure equivalent to the natural area form on \(\cM\). This makes \((\cM, \fA, \mu)\) a Lebesgue space. Let \(\cF : \cM \to \cM\) be defined almost everywhere. The dynamical system \((\cM, \fA, \mu, \cF)\) satisfies the following assumptions (H1)-(H7), which are particularly relevant in the context of hyperbolic billiards \cite[Section 6.3]{MR2229799}.  

\begin{itemize}
\item[(H1)](Smoothness with singularities) We assume that $\mathcal{F}$ is an invertible and differentiable map defined everywhere except for singularities. More precisely, we assume that there exist two countable unions \(\cS_1, \cS_{-1}\) of closed curves  such that \(\cF\) is a diffeomorphism \(\cM \setminus \cS_1 \to \cM \setminus \cS_{-1}\).
\end{itemize}
This definition implies that \(\cF^n\) and \(\cF^{-n}\) can only be defined, respectively, on \(\cM \setminus \cS_n\) and \(\cM \setminus \cS_{-n}\), where, for \(n \ge 2\),
\begin{equation}\label{eq:eq-future-sing}
\mathcal{S}_{n} = \mathcal{S}_{n-1} \cup \mathcal{F}^{-1}(\mathcal{S}_{n-1}), \qquad \mathcal{S}_{-n} = \mathcal{S}_{-n+1} \cup \mathcal{F}(\mathcal{S}_{-n+1}).
\end{equation}
The sets where some future, respectively, past iterate of $\mathcal{F}$ is singular are denoted by \(\mathcal{S}_{\infty} = \cup_{n = 0}^{\infty}\mathcal{S}_{n}\) and \(\mathcal{S}_{-\infty} = \cup_{n = 0}^{\infty}\mathcal{S}_{-n}\). For any $x \in \mathcal{M}$ we call \emph{local stable manifold} (LSM) at \(x\) any $C^{1}$ open curve \(W_0^s\) containing $x$ and not intersecting \(\cS_{\infty}\), such that \(\lim_{n \to \infty}|\cF^n (W_0^s)| =0\). \emph{Local unstable manifolds} (LUMs) are defined similarly, with \(\cF^{-1}\) instead of 
\(\cF\). 
\begin{itemize}
\item[(H2)](Stable and unstable manifolds) We assume that for almost every point $x\in \cM$ there exist a LSM $W^{s} \ni x$ and a LUM $W^{u} \ni x$. In particular, there exist mod \(0\) partitions \(\xi^s\) and \(\xi^u\) of \(\cM\) into LSMs and LUMs respectively. For any \(x \in \cM\), we denote by \(W^s(x)\) (resp.\ \(W^u(x)\)) the unique element \(W^s\) of \(\xi^s\) (resp.\ \(\xi^u\)) such that \(x \in W^s\). (If \(x\) does not possess any LSM we set \(W^s(x) =\emptyset\) and analogously for LUMs). We also assume that the lengths of \(W^s\) and \(W^u\) are uniformly bounded and for all \(n \in \bN\) and \(x \in \cM\),
\begin{equation}\label{eq:non-unif-hyp}
\begin{split}
    \cF^n(W^s(\cF^{-n}(x))) \subseteq W^s(x), \qquad & \cF^{-n}(W^u(\cF^{n}(x))) \subseteq W^u(x),\\[4pt]
    \lim_{n \to \infty} |\cF^{n}(W^s(\cF^{-n}(x)))| = 0, \qquad & \lim_{n \to \infty} |\cF^{-n}(W^u(\cF^{n}(x)))| = 0.
\end{split}
\end{equation}
\end{itemize}
For any \(\cC^1\) curve \(W\), we denote by \(\bold m_W\) the one-dimensional Lebesgue measure on \(W\) and by \(\widehat{\bold m}_{W}\) the same measure, but normalized so that \(\widehat{\bold m}_{W}(W) = 1\).
\begin{itemize}
\item[(H3)](Measurability) 
We assume that the partitions \(\xi^s\) and \(\xi^u\) are measurable. By classical results \cite{MR47744}, \(\xi^u\) induces 
a disintegration \(\mu_{\xi^u}\) of \(\mu\) in conditional probability measures \(\nu_{W^u}\). We assume that \(\nu_{W^u}\) is equivalent to \(\widehat{\bold m}_{W^u}\), uniformly in \(W^u \in \xi^u\). More precisely, there exists \(C>1\) such that, for a.e.\ \(W^{u} \in \xi^u\),
\begin{equation}\label{eq:equivalent-condiditonal-measure}
C^{-1} \le \frac{d\nu_{W^u}}{d\widehat{\bold m}_{W^u}} \le C.
\end{equation}
The same is true for the partition $\xi^s$ into stable manifolds.
\end{itemize}
(H3) means that there exists a measure space \((\cM_{\xi_u}, \fA_{\xi^u}, \mu_{\xi_u})\) and, for a.e.\ \(W^u \in \xi^u\), measures \(\nu_{W^u}\) supported on \(W^u\) such that, for any \(A \in \fA\),
\begin{equation}\label{eq:H3-meaning}
\mu(A) = \int_{\cM_{\xi^{u}}} \nu_{W}(A\cap W) \, d\mu_{\xi^{u}}(W).
\end{equation}

We now define Cantor rectangles following \cite[Section 7.11]{MR2229799}. For any \(x,y \in \cM\), we set
\[
[x,y] = W^s(x) \cap W^u(y).
\]
A subset \(\cantor\) is called a \emph{Cantor rectangle} if, for any \(x,y\in \cantor\), we have \(\emptyset \neq [x,y] \subset \cantor\). We also require that for each Cantor rectangle \(\cantor\), there exists \(c(\cantor)>0\) such that, for all \(x \in \cantor\), we have \(|W^u(x)| \ge c(\cantor)\). For any Cantor rectangle \(\cantor\), set \(\cantor^0 = \cantor\) and define inductively
\[
\cantor^{n+1} = \bigcup_{x \in \cantor^n} \left( W^s(x) \cup W^u(x) \right).
\]
Finally, set \(\underline \cantor = \cup_{n=0}^{\infty}\cantor^n\). We call \(\underline \cantor\) an \emph{enlarged Cantor rectangle}. Notice that \(\underline \cantor\) is saturated by LSMs and LUMs in the sense that \(\underline \cantor = \cup_{x \in \underline \cantor}W^s(x) = \cup_{x \in \underline \cantor}W^u(x)\).

\begin{itemize}
    \item [(H4)](Cantor rectangles) There exists a countable sequence \(\{\cantor_k\}_{k=1}^{\infty}\) of Cantor rectangles such that \(\cM = \cup_{k=1}^{\infty}\cantor_k\) mod 0. Moreover,\footnote{We are also assuming that \(\cantor_k\) and \(\underline \cantor_k\) are measurable.} we assume that \(\mu(\underline \cantor_k)<\infty\), for all \(k \in \bN\).
\end{itemize}
Without loss of generality, we assume that \(\mu(\cantor_k) >0\) for all \(k\) as well.
\begin{itemize}
\item[(H5)](Distortion)
For any \(\cC^1\) curve \(W\) and \(x\in W\), let \(\cT_{x}W\) be the tangent space to \(W\) at \(x\). For all \(v \in \cT_{x}W\) and \(n \in \mathbb Z\), we denote by
\[
    \cJ_{W} \cF^{n} = \frac{\|D_x \cF^{n}v\|}{\|v\|}
\]
the expansion/contraction coefficient along \(W\) at \(x\) (whenever defined). We assume that there exists $D>1$ such that, for a.e.\ $W \in \xi^u$ and all $y, z \in W$, \(n \in \bN\),
\begin{equation}
 \label{eq:dis1}
 D^{-1} \le \biggl | \frac{\mathcal{J}_{W}\cF^{-n}(y)}{\mathcal{J}_{W}\cF^{-n}(z)} \biggr | \le D.
\end{equation}
\end{itemize}

\begin{itemize}
\item[(H6)](Absolute continuity) Let $W^{1}, W^{2} \in \xi^u$ and denote, for $i = 1,2$,
\begin{equation*}
W^{i}_{*} = \{ x \in W^{i} \,:\, W^{s}(x)\cap W^{3-i} \ne \emptyset \}.
\end{equation*}
We assume that, for every $x \in W^{1}_{*}$, $W^{s}(x)\cap W^2$ consists of a single point, which we call $\bold h(x)$. This defines a map $\bold h: W_{*}^{1} \to W_{*}^{2}$, called \emph{holonomy map}. We assume that the measure \(\bold h^{*}\bold m_{W_{*}^2}\) given by the pull-back of \(\bold m_{W_{*}^2}\) by \(\bold h\) is ``uniformly equivalent'' to \(\bold m_{W_{*}^1}\). This means that there exists \(\tilde D >1\) such that, denoting \(J \bold h = d( \bold h^{*} \bold m_{W_{*}^2})/d\bold m_{W_{*}^1} \), 
\begin{equation}\label{eq:bound-holonomy}
        \tilde D^{-1}\le J \bold h \le \tilde D.
\end{equation}
The same is true for \(W^1, W^2 \in \xi^s\) and considering the map obtained by sliding along unstable manifolds.
\end{itemize}

As a direct consequence of \eqref{eq:H3-meaning} and (H6), for any \(B \subseteq W \in \xi^u\) such that the sets below are measurable,
\begin{equation}\label{eq:meaning-H6}
    \bold m_{W}(B) > 0 \Longleftrightarrow \mu\biggl( \bigcup_{x \in B} W^s(x)\biggr) >0.
\end{equation}
The last assumption is crucial for our technique of proving the K-property.
\begin{itemize}
\item[(H7)](Recurrence) We assume that the dynamical system $(\mathcal{M},  \fA, \mu, \cF)$ is recurrent: there exists no \(A \in \fA_+\) such that the sets \(\{\cF^k (A)\}_{k \in \bN_0}\) are disjoint mod 0.
\end{itemize}

We are now ready to state the main result of this paper.

 \begin{theorem}\label{thm:strth} (K-Decomposition Theorem) Let $\mathcal{F}: \mathcal{M} \rightarrow \mathcal{M}$ be a map satisfying hypothesis (H1)-(H7). Then $\mathcal{F}$ admits a K-decomposition. Precisely, there exists a finite or countable partition \(\mathcal{M} = \bigcup_{i }\mathcal{E}_{i}\), such that 
 \renewcommand{\labelenumi}{\alph{enumi})}
 \begin{enumerate}
 \item each $\mathcal{E}_{i}$ is $\mathcal{F}$-invariant, $\mu(\mathcal{E}_{0}) = 0$ and $\mu(\mathcal{E}_{i}) > 0$ for $i \ge 1$;
 \item the restriction $\mathcal{F}: \mathcal{E}_{i} \rightarrow \mathcal{E}_{i}$ is ergodic for every $i \ge 1$.
 \\
 Furthermore, for every $i \ge 1$ there exists \(k_i \in \bN\) and a finite partition
 \begin{equation}
 \mathcal{E}_{i} = \mathcal{E}_{i}^{0} \cup ... \cup \mathcal{E}_{i}^{k_{i}-1}
 \end{equation}
 such that
  \item $\mathcal{F}(\mathcal{E}_{i}^{j}) = \mathcal{E}_{i}^{j+1}$ and $\mathcal{F}(\mathcal{E}_{i}^{k_{i}-1}) = \mathcal{E}_{i}^{0}$;
 \item the map $\mathcal{F}^{k_{i}}: \mathcal{E}_{i}^{j} \rightarrow \mathcal{E}_{i}^{j}$ is K-mixing.
 \end{enumerate}
 \end{theorem}
 
 \begin{corollary}
 \label{corlStr}
 Let $\mathcal{F}: \mathcal{M} \rightarrow \mathcal{M}$ be a map satisfying hypothesis (H1)-(H7). Assume that $\mathcal{F}^{n}$ is ergodic for any $n \ge 1$. Then $\mathcal{F}$ is K-mixing.
 \end{corollary}
 \proof
 By ergodicity of \(\cF\) there exists only one ergodic component. Hence, according to the notation of Theorem \ref{thm:strth}, $\mathcal{M} = \mathcal{E}_{1}$ mod $0$. Moreover, by parts b) and c) of the same theorem and the ergodicity of all the powers of \(\cF\) we have that $\mathcal{E}_{1} = \mathcal{E}_{1}^{0}$ mod $0$. Finally, by point d), $\mathcal{F}: \mathcal{E}_{1}^{0} \to \mathcal{E}_{1}^{0}$ is K-mixing.
 \qed

\section{Proof of the main theorem}\label{sec:proof}
Recall that we are considering a dynamical system on \((\cM, \fA, \mu)\) where \(\cM\) is a Riemannian manifold, \(\fA\) is the completion of its Borel \(\sigma\)-algebra and \(\mu\) is equivalent to the area form. We denote by \(\fS\) the $\sigma$-algebra generated by \(\xi^s\),
\begin{equation}\label{eq:stable-sigma}
\begin{split}
\fS = \{A \in \fA \,:\, A = B \text{ mod }0 \text{ for }& \text{some \(B\in \fA\)}\\
&\text{which is a union of elements of \(\xi^s\)}\}.
\end{split}
\end{equation}
We have that \(\fS\) is a complete sub-$\sigma$-algebra of \(\fA\). We start by proving the first two properties of \eqref{def:K} for this choice of \(\fS\). For a partition \(\xi\) of \(\cM\) we denote \(\cF \xi = \{\cF(W) \,:\, W \in \xi\}\).

\begin{lemma}\label{lem:first-prop-K}
\(\cF \fS \supseteq \fS\).
\end{lemma}
\proof
For any \(W \in \xi^s\), set
\[
E_W = W \triangle \bigcup_{x \in W} \cF(W^s(\cF^{-1}(x))).
\]
By the first line of \eqref{eq:non-unif-hyp}, for all \(W \in \xi^s\) we have that \(W \supseteq \cup_{x \in W} \cF(W^s(\cF^{-1}(x)))\). Moreover, \(W \setminus \cup_{x \in W} \cF(W^s(\cF^{-1}(x)))\) consists of all points \(x\in W\) such that \(\cF^{-1}(x)\) does not have a LSM. (Indeed, if \(W^s(\cF^{-1}(x))\) is non-empty, it contains \(\cF^{-1}(x)\) so that \(\cF (W^s(\cF^{-1}(x)))\) contains \(x\)). Take a set \(A \in \fS\) and let \(A_0 = A\) mod \(0\) such that \(A_0\) is a union of LSMs. If \(\cup_{W \in \xi^s, \text{ }W \subseteq A_0}\bigr(W \setminus \cup_{x \in W} \cF(W^s(\cF^{-1}(x)))\bigl)\) were non-empty mod \(0\), a positive measure of points would not have a LSM (here we also use the non-singularity of \(\cF\)). This contradicts (H2). In summary, we have shown that \(\cup_{W \in \xi^s, \text{ }W\subseteq A_0}E_W = \emptyset\) mod \(0\). Therefore, any measurable union of LSMs is equal, modulo a null set, to a union of elements of \(\cF\xi^s\). Recalling  \eqref{eq:stable-sigma}, the last assertion proves the statement.
\qed

\begin{lemma}\label{lem:second-K-property} $\bigvee_{n=0}^{\infty}\mathcal{F}^{n}\fS = \mathfrak A$.
\end{lemma}
\proof
The \(\subseteq\) inclusion is a consequence of the fact that \(\cF^{-1}\) is measurable and that \(\fS \subset \fA\). For the \(\supseteq\) inclusion, it is sufficient to show that the collection \(\mathcal R =\bigl\{ B(p, r) \,:\, p\in \cM, r>0 \bigr\}\) of open balls (which is a generator of \(\fA\)) belongs to \(\bigvee_{n=0}^{\infty}\mathcal{F}^{n}\fS\). Let \(B \in \cR\) the ball of radius \(r\) centered at \(p\). Set, for \(n \in \bN_0\), 
\[
B_{n} = \bigcup_{x \in B} \cF^{n}(W^s(\cF^{-n}(x)))  \in \cF^{n}\fS.
\]
Since \(\mu\) is locally finite, \(\mu(B) < \infty\). Observe that the convergence stated in the second line of \eqref{eq:non-unif-hyp} is uniform outside of a set of arbitrary small measure: for each \(\ve >0 \), there exist \(M_{\ve} \subseteq B\) with \(\mu(M_{\ve}) \le \ve\) and \(n_{\ve} \in \bN\) such that on \(B \setminus M_{\ve}\) \(|\cF^{n}(W^s(\cF^{-n}(x)))| \le \ve\) for any \(n \ge n_{\ve}\). Hence, for any such \(n\),
\[
\begin{split}
\mu(B_{n} \setminus B) \le \mu(M_{\ve}) + \mu(B_{n} \setminus (B\cup M_{\ve})) &\le \ve + \mu(\{q\in \cM \,:\, r \le |q-p| \le r + \ve\})\\
&\le \ve + o(1),
\end{split}
\]
as \(\ve \to 0\). In the last inequality we have used that \(\mu\) is locally finite and equivalent to the area. It follows that \(\mu ((\cap_n B_{n}) \setminus B) = 0\). By the completeness of \(\bigvee_{n=0}^{\infty}\mathcal{F}^{n}\fS\), this shows that \(B \in \bigvee_{n=0}^{\infty}\mathcal{F}^{n}\fS\) and we conclude the proof.
\qed
Denote by \(\fT = \bigcap_{k = 0}^{\infty} \cF^{-k}\fS\) the \textit{tail-$\sigma$-algebra} relative to \(\fS\). As can be deduced by the definition \eqref{def:K}, the tail is particularly relevant for the K-decoposition of Theorem \ref{thm:strth}. We start by observing an invariance property.
\begin{lemma}\label{lem:tail-inv} 
    \(\fT = \cF \fT = \cF^{-1}\fT\).
\end{lemma}
\proof
 It is sufficient to observe that, by Lemma \ref{lem:first-prop-K}, \(\{\cF^{-k}\fS\}_{k}\) is a non-increasing sequence of $\sigma$-algebras.
\qed
Recall  the Cantor rectangles \(\cantor_k\) and the enlarged Cantor rectangles \(\underline \cantor_k\) from (H4). For any \(k, d \in \bN\), denote by \(\cF_{d,k}: \underline\cantor_k \to \underline \cantor_k\) the return map of \(\cF^d\) to \(\underline \cantor_k\). More precisely, for a.e.\ \(x \in \underline \cantor_k\), set
\begin{equation}\label{eq:return-maps}
\begin{split}
\cF_{d,k}(x) = \cF^{d n(x)}(x) \quad & \text{where} \quad n(x) = \min\{n \in \bN \,:\, \cF^{d n}(x) \in \underline \cantor_k\},\\
\cF_{d,k}^{-1}(x) = \cF^{-d m(x)}(x) \quad & \text{where} \quad m(x) = \min\{m \in \bN \,:\, \cF^{-d m}(x) \in \underline \cantor_k\}.
\end{split}
\end{equation}
By the recurrence hypothesis (H7), \(n(x)\) and \(m(x)\) are a.s.\ finite and the above maps are a.s.\ well-defined. Given a $\sigma$-algebra \(\fB\subset \fA\) and \(A \in \fA\), we denote by \(\fB \cap A = \{B \cap A \,:\, B \in \fB\}\). Recall also the definition of \(\bold m_{W}\) and \(\widehat{\bold m}_{W}\) after \eqref{eq:non-unif-hyp}.
\begin{lemma}\label{lem:hopf-on-cantor}
    For any \(k \in \bN\), \(\bigl(\underline \cantor_k, \fA \cap \underline \cantor_k, \mu, \cF_{d,k}\bigr)\) is ergodic.
\end{lemma}
\proof
By definition of enlarged Cantor rectangle, if \(x \in \underline \cantor_k\), then \(W^u(x)\cup W^s(x) \subset \underline \cantor_k\). Moreover, by the first line of \eqref{eq:non-unif-hyp} (and the non-singularity of \(\cF\) and its inverse), for all \(n \in \bN\) and \(x\), we have that \(\cF^n (W^s(x)) \subseteq W^s (\cF^n (x))\) and \(\cF^{-n} (W^u(x)) \subseteq W^u (\cF^{-n} (x))\). The last two considerations imply \(n(x)\) and \(m(x)\) are finite and constant on \(W^s(x)\) and \(W^u(x)\) for a.e.\ \(x \in \underline \cantor_k\). Therefore for a.e.\ \(W^s \subset \underline \cantor_k\) and any \(z_1, z_2 \in W^s\),
\begin{equation}\label{eq:diamter-return-map}
\begin{split}
\lim_{n \to \infty} |\cF_{d,k}^n(z_1) - \cF_{d,k}^{n}(z_2)| \le \lim_{n \to + \infty} |\cF^n (W^{s})|  =0.
\end{split}
\end{equation}
By the analogous argument, for a.e.\ \(W^u \subset \underline \cantor_k\) and any \(w_1, w_2 \in W^u\),
\begin{equation}\label{eq:diamter-return-map1}
\begin{split}
&\lim_{n \to \infty} |\cF_{d,k}^{-n}(w_1) - \cF_{d,k}^{-n}(w_2)| \le \lim_{n \to + \infty} |\cF^{-n} (W^{u})|  =0.
\end{split}
\end{equation}
To conclude the proof, we apply a standard Hopf argument. We refer to \cite[Chapter 6]{MR2229799} or \cite[Proposition 10.2]{MR1346498} for the argument in full details. For \(f : \underline \cantor_k  \to \bR\), \(x \in \underline \cantor_k\), set
\[
f^{+}(x) = \lim_{n\to \infty} \frac{\sum_{j=0}^{n-1}f(\cF_{k,d}^{j}(x))}{n}, \qquad f^{-}(x) = \lim_{n\to \infty} \frac{\sum_{j=0}^{n-1}f(\cF_{k,d}^{-j}(x))}{n}.
\]
By (H4), \(0<\mu(\underline \cantor_k) <\infty\). Therefore, to prove ergodicity, it is sufficient to show that \(f^+\) is constant mod \(0\) for all \(f\) belonging to a dense set in \(L^1(\underline \cantor_k)\), which for us would be the set of uniformly continuous functions. By the Birkhoff ergodic theorem for almost all \(x \in \underline \cantor_k\) the limits \(f^+ (x)\) and \(f^{-}(x)\) exist and it's a basic fact that they are equal mod \(0\). Moreover, by \eqref{eq:diamter-return-map} and \eqref{eq:diamter-return-map1} and the fact that \(f\) is uniformly continuous, we have that \(f^+\) is constant on a.e\ stable manifold  and \(f^{-}\) is constant on a.e.\ every unstable manifold in \(\underline \cantor_k\).

Recall the definition of \(\cantor_k^n\) before (H4). By the paragraph above and (H3) it follows that a.e.\ LSM and LUM \(W \in \cantor_k^1\) has the property that \(f^+ = f^-\)  for a.e.\ point w.r.t.\ the internal Lebesgue measure \(\bold m_W\). 
Call \(\bold T_u\) the sets of LUMs with the above property and such that \(f^{-}\) is constant on them. Let \(W_*^u \in \bold T_u\). We have that \(f^{+}(x) = f^{-}(x) = A \in \bR\) for \(\bold m_{W_*^u}\)-almost any \(x \in W_*^u\). Moreover, by definition of Cantor rectangle, all the LSMs \(W^s \subset  \cantor_k^1\) intersect \(W_*^u\). Hence, using \eqref{eq:meaning-H6} (absolute continuity of the stable foliation) and the fact that \(f^{+}\) is constant on a.e.\ LSM, one deduces that \(f^{+} = A\) on a.e. LSM in \(\cantor_k^1\). By a similar argument, we have that \(f^- = A\) on a.e. LUM of \(\cantor_k^1\). Because \(f^+ = f^{-}\) a.e., the above shows that \(f^+\) is constant mod \(0\) on \(\cantor_k^1\). Iterating this argument, one obtains that \(f^+\) is constant mod \(0\) on \(\cantor_k^n\) for each \(n\). This concludes the proof.
\qed
For any \(d \in \bN\), \(B \in \fA\) denote by $\mathfrak{I}_d$ the sub-$\sigma$-algebra of \(\cF^{d}\)-invariant sets, 
\[
\mathfrak{I}_d = \{A\in\mathfrak{A} \,:\, \mathcal{F}^{-d}(A) = A \text{ mod 0} \},
\]
and by \(\Erg_d(B)\) the \textit{\(\cF^d\)-ergodic hull} of \(B\), that is, the smallest set in \(\fI_d\) containing \(B\):
\begin{equation}\label{eq:ergodic-hull}
    \Erg_d(B) = \bigcup_{j=-\infty}^{\infty} \cF^{dj}(B).
\end{equation}
We say that \(P\) is an atom of a sub-$\sigma$-algebra \(\fB \subseteq \fA\) if \(\mu(P)>0\) and there exists no \(E \subseteq P\), \(E \in \fB\) such that \(0<\mu(E) < \mu(P)\). We say that \(\fB\) is atomic if it is generated by atoms. Recall also the notation \(\fB_+\) for the sets of \(\fB\) of positive measure.

\begin{lemma}\label{lem:ergodic-components}
    For any \(d \in \mathbb Z\setminus\{0\}\) and \(k \in \bN\), \(\Erg_d(\underline \cantor_k)\) is an atom of \(\fI_d\). In particular, \(\fI_d\) is atomic.
\end{lemma}
\proof
Fix any \(d \in \mathbb Z \setminus \{0\}\). As we already observed, by definition of ergodic hull \(\Erg_d(\underline \cantor_k) \in \fI_d\). Let \(A \in \fI_{d}\), \(A \subseteq \Erg_d(\underline \cantor_k)\). Recalling \eqref{eq:return-maps} for the first-return maps \(\cF_{d,k}\), by invariance of \(A\),
\[
  \cF_{d,k}^{-1}(A \cap \underline \cantor_k) = A \cap \underline \cantor_k \quad \text{mod} \text{ }0.
\]
Hence, by Lemma \ref{lem:hopf-on-cantor}, \(A\) equals \(\Erg_d(\underline \cantor_k)\) or \(\emptyset\) mod \(0\). Therefore any positive-measure subset \(A \in \fI_d\) of \(\Erg_d(\underline \cantor_k)\) coincides mod 0 with \(\Erg_d(\underline \cantor_k)\), proving the first part of the statement. As for the second part, by (H4),
 \[
 \cM \subseteq \bigcup_{k=1}^{\infty}\cantor_k \subseteq \bigcup_{k=1}^{\infty} \Erg_d( \underline \cantor_k) \quad \text{mod }0.
 \]
In other words, the phase space is covered by atoms, implying the statement.
\qed
In the next lemma we argue that the ergodic components belong to the tail.
\begin{lemma}\label{lem:K-and-ergodicity}
\(\fI_{1} \subseteq \fT\).    
\end{lemma}
\proof
By Lemma \ref{lem:ergodic-components}, it is sufficient to show that \(\Erg_1(\underline \cantor_k) \in \fT\), for any \(k \in \bN\). Since \(\underline \cantor_k\) is a union of LSMs, recalling \eqref{eq:stable-sigma}, \(\underline \cantor_k \in \fS\). By recurrence (H7), for any \(n \in \bN\), \(\Erg_1(\underline \cantor_k) = \bigcup_{j=-\infty}^{\infty}\cF^{-j}(\underline \cantor_k) = \bigcup_{j=n}^{\infty}\cF^{-j}(\underline \cantor_k)\) mod \(0\). Therefore, by Lemma \ref{lem:first-prop-K}, \(\Erg_1(\underline \cantor_k)\in \cF^{-n}\fS\) for each \(n \in \bN\) and consequently \(\Erg_1(\underline \cantor_k) \in \fT\).
\qed

\begin{remark}\label{rm:k-mixing-implies-ergodicity}
    It is noteworthy that the above inclusion is not true in general for non-recurrent dynamical systems (see \cite[Corollary in Section 2]{MR181737}). In particular, a system can be K-mixing but completely dissipative, hence not ergodic (see \cite[Section 1.1]{MR1450400}).
\end{remark}
In the following, when we say that \(x\) is Lebesgue density point for \(A \in \fA\) w.r.t.\ $\bold m_{W^{u}(x)}$ we mean that
\[
\lim_{r \to 0^+}\frac{\bold m_{W^u(x)} (A \cap B(x,r))}{\bold m_{W^u(x)}(B(x,r))} = 1,
\]
where \(B(x,r)\) is a ball of radius \(r\) centered at \(x\). We have the following elementary observation.
\begin{lemma}
\label{lem:density}
For any $A \in \fA$, \(\mu\)-a.e.\ point $x \in A$ is a Lebesgue density point of \(A\) w.r.t.\ $\bold m_{W^{u}(x)}$.
\end{lemma}
\proof
Call 
\[
\begin{split}
B = \left\{x \in A \,:\, x \text{ is not a density point of } W^{u}(x)\cap A \text{ w.r.t.\ } \bold m_{W^{u}(x)} \right\}.
\end{split}
\]
Every time \(W^u \cap A\) is measurable w.r.t.\ the Borel \(\sigma\)-algebra on \(W^u \in \xi^u\) (this occurs $\mu_{\xi^u}$-almost everywhere), we have that \( W^u\cap B\) is measurable. Therefore, by (H3), \(B\) is measurable and, recalling \eqref{eq:equivalent-condiditonal-measure}, for some \(C>0\),
\[
\mu(B) = \int_{\cM_{\xi^u}} \nu_{W^u}(B \cap W^u) \, d\mu_{\xi^u}(W^u) \le C \int_{\cM_{\xi^u}} \widehat{\bold m}_{W^u}(B \cap W^u) \, d\mu_{\xi^u}(W^u).
\]
On the other hand, by the definition of \(B\), we have \(\widehat{\bold m}_{W^u}(B \cap W^u) = 0\). Therefore, by the equation above, \(\mu(B) = 0\), which is the sought claim.
\qed
We now need a distortion estimate, which we prove in Lemma \ref{lem:distortion}. We start with the following easy consequence of assumption (H5). Let \(D > 1\) be the constant from \eqref{eq:dis1}. 

\begin{lemma}\label{lem:distortion-bound}
For any \(A \in \fA\) , \(x \in \cM\), \(n \in \bN\) and connected curve \(W_n \subseteq W^u(\cF^n(x))\),
\[
 \begin{split}
 D^{-1} \le \frac{\widehat{\bold m}_{W_n}\bigl(\cF^n(A) \cap W_n\bigr)}{\widehat{\bold m}_{\cF^{-n}(W_n)}(A \cap \cF^{-n}(W_n))}\le D.
 \end{split}
\]
\end{lemma}
\proof
By the change of variable formula, we have
\begin{equation}\label{eq:change-variable}
  \bold m_{\cF^{-n}(W_n)}(A \cap \cF^{-n}(W_n))=\int_{W_n \cap \cF^n(A)} \mathcal{J}_{W_n} \cF^{-n}(x) \, d\bold m_{W_n}(x).
\end{equation}
Setting \(A = \cF^{-n}(W_n)\) in \eqref{eq:change-variable}, we obtain
\[
     \bold m_{\cF^{-n}(W_n)} (\cF^{-n}(W_n))= \int_{W_n} \mathcal{J}_{W_n} \cF^{-n}(x) \, d \bold m_{W_n}(x).
\]
Since \(\cJ_{W_n} \cF^{-n}: W_n \to \bR^+\) is a continuous function and \(W_n\) is connected by hypothesis, there exists \(x_0 \in W_n\) such that
\[
    \cJ_{W_n} \cF^{-n}(x_0) =  \frac{ \bold m_{\cF^{-n}(W_n)} (\cF^{-n}(W_n))}{\bold m_{W_n}(W_n)}.
\]
By the last equation and \eqref{eq:dis1}, one has that, for some \(D>1\), on \(W_n\),
\[
D^{-1}\frac{ \bold m_{\cF^{-n}(W_n)} (\cF^{-n}(W_n))}{\bold m_{W_n}(W_n)} \le \cJ_{W_n} \cF^{-n} \le D \frac{ \bold m_{\cF^{-n}(W_n)} (\cF^{-n}(W_n))}{\bold m_{W_n}(W_n)}.
\]
Combining the above equation with \eqref{eq:change-variable}, we obtain the statement.
\qed
The next lemma is the main distortion estimate.
\begin{lemma}
\label{lem:distortion}
For any \(A\in \fA\), a.e.\ \(x_0 \in A\) and any \(\ve>0\), there exists $\bar{n}\in \mathbb{N}$ such that, for all $n \ge \bar{n}$, 
\[
\widehat{\bold m}_{W^u(\cF^n(x_0))} \bigl(W^u(\cF^n(x_0)) \setminus \cF^n(A)\bigr) \le \ve.
\]
\end{lemma}
\proof
By Lemma \ref{lem:density}, a.e.\ \(x_0 \in A\) is a density point w.r.t.\ \(\bold m_{W^u(x_0)}\). Since this is a typical conditions, we can assume that \(W^u(x_0)\) is non empty. For notational convenience, set \(W_n = W^u(\cF^n(x_0))\). By \eqref{eq:non-unif-hyp},
\[
    \lim_{n \to \infty} |\cF^{-n}(W_n)| = 0. 
\]
Fix \(\ve >0\). By the equation above, the fact that \(x_0\) is a density point of \(A\) for the measure \(\bold m_{W^{u}(x_{0})}\) and since \(\cF^{-n}(W_n)\) is a curve containing \(x_0\), we have that 
\begin{equation*}
\lim_{n \rightarrow \infty} \frac{\bold m_{W^{u}(x_{0})}(A^c \cap \cF^{-n}(W_n))}{\bold m_{W^{u}(x_{0})}(\cF^{-n}(W_n))} = 0.
\end{equation*}
Let \(D>1\) be the constant appearing in the statement of Lemma \ref{lem:distortion-bound}. By the equation above, there exists $\bar{n} \in \bN$ so large that, for all $n \ge \bar{n}$,
\begin{equation}\label{eq:density-point}
\begin{split}
\widehat{\bold m}_{\cF^{-n}(W_n)}(A^c \cap \cF^{-n}(W_n))= \frac{\bold m_{W^{u}(x_{0})}(A^c \cap \cF^{-n}(W_n))}{\bold m_{W^{u}(x_{0})}(\cF^{-n}(W_n))}\le \frac{\ve}{ D}.
\end{split}
\end{equation}
Hence, using that the map is invertible and Lemma \ref{lem:distortion-bound}, we obtain, for all $n \ge \bar{n}$,
\[
   \widehat{\bold m}_{W_n}\bigl(W_n \setminus \cF^n(A) \bigr) =  \widehat{\bold m}_{W_n}\bigl(\cF^n(A^c) \cap W_n \bigr) \le D \widehat{\bold m}_{\cF^{-n}(W_n)}(A^c \cap \cF^{-n}(W_n)) \le \ve.
\]
This concludes the proof.
\qed

The following lemma highlights a crucial property of the sets in the tail $\sigma$-algebra.
\begin{lemma}\label{lem:global-manifolds}
    For any \(A \in \fT\) and \(n \in \mathbb Z\), there exists a set \(A^n\) which is a union of LSMs such that
    \[
    \cF^{n}(A) = A^n \quad \text{mod }0.
    \]
Moreover, for any sequence of sets \(A^n\) as above and a.e.\ \(x \in A\) we have that \(W^s(\cF^n(x)) \subseteq A^n\) for all \(n \in \mathbb Z\).
\end{lemma}
\proof
Take $n \in \Z$.
For any \(B \in \cF^{-n}\fS\), there exists \(B_0 \in \fS\) such that \(B = \cF^{-n}(B_0)\). Therefore, recalling \eqref{eq:stable-sigma}, there exists \(\Xi\subseteq \xi^s\) such that \(\cF^{n}(B) = B_0 = \bigcup_{W \in \Xi} W\) mod \(0\). Hence, the first part of the statement follows from the fact that any \(A \in \fT\) belongs to \(\cF^{-n}\fS\) for all \(n\in \mathbb Z\). For the second part, if the assertion weren't true, there would exist \(\bar n \in \mathbb Z\) and a set of positive measures \(N \subseteq A\) such that \(W^s(\cF^{\bar n} (x))\) is not contained in \(A^{\bar n}\) for any \(x \in N\). By the non-singularity of \(\cF\), in this case, there exists a positive-measure set \(N' \subseteq \cF^{\bar n}(A)\) such that \(W^s(x)\) is not a subset of \(A^{\bar n}\) for each \(x \in N'\). Since \(A^{\bar n}\) is a union of LSMs, every \(x \in A^{\bar n}\) satisfies \(W^s(x) \subseteq A^{\bar n}\). It follows that \(N' \subseteq \cF^{\bar n}(A) \setminus A^{\bar n}\), contradicting the first part of the statement. 
\qed
For any \(W\in \xi^u\) and \(B \subseteq W\), we denote by \(\brush_{W}  (B)\) the \textit{brush of local stable manifolds based on \(B\)},
\[
     \brush_{W}(B) = \bigcup_{x \in B }W^s(x).
\]
The following result is a consequence of absolute continuity.
\begin{lemma}\label{lem:lebw-vsmu}
For any \(\cantor_k\) there exists \(C_k>0\) such that, for any \(W \in \xi^u\) and measurable \(B \subseteq W\) and $H \subseteq \brush_{W}(B) \cap \cantor_k$,
\[
   \mu ( H ) \le C_k \widehat{\bold m}_{W}(B \cap W) .
\]
\end{lemma}
\proof
Applying \eqref{eq:H3-meaning} to \(H = H \cap \brush_{W}(B) \cap \cantor_k\), we write
\[
\mu(H) = \int_{\{\widetilde W \in \cM_{\xi^u} \,:\, \widetilde W \cap \underline \cantor_k \neq \emptyset\}} \nu_{\widetilde W} ( H\cap \brush_{W}(B) \cap \cantor_k \cap \widetilde W) \, d\mu_{\xi^u}(\widetilde W).
\]
By the same hypothesis (H3), the conditional probability measures \(\nu_{\widetilde W}\) are equivalent to \(\widehat{\bold m}_{\widetilde W}\). Therefore, for some \(C>0\),
\[
   \mu(H)\le C \int_{\{\widetilde W \in \cM_{\xi^u} \,:\, \widetilde W \cap \underline \cantor_k \neq \emptyset\}} \widehat{\bold m}_{\widetilde W} ( H\cap \brush_{W}(B) \cap \cantor_k \cap \widetilde W) \, d\mu_{\xi^u}(\widetilde W). 
\]
For any \(\widetilde W \in \xi^u\), we consider the set \(W_{*}\) of points \(x \in W\) such that \(W^s(x) \cap \widetilde W \neq  \emptyset\) and let \(\bold h: W_{*} \to \widetilde W\) be the holonomy map. By (H6),
\[
\begin{split}
      \widehat{\bold m}_{\widetilde W}( H\cap \brush_{W}(B) \cap \cantor_k \cap \widetilde W) &= \frac{\bold m_{\widetilde W}( H\cap \brush_{W}(B) \cap \cantor_k \cap \widetilde W)}{\bold m_{\widetilde W}(\widetilde W)} \\
      &=  \frac{\bold m_{\widetilde W}(H\cap\bold h(B\cap W_{*}) \cap \cantor_k \cap \widetilde W)}{\bold m_{\widetilde W}(\widetilde W)} \\
      &\le  \frac{\bold m_{\widetilde W}(\bold h(B \cap W_{*}) \cap \widetilde W)}{\bold m_{\widetilde W}(\widetilde W)} \le  \frac{\tilde D \bold m_{W}( B \cap W)}{\bold m_{\widetilde W}(\widetilde W)}.
\end{split}
\]
Hence, dividing and multiplying by \(\bold m_W(W)\),
\[
\widehat{\bold m}_{\widetilde W}( H\cap \brush_{W}(B) \cap \cantor_k \cap \widetilde W) \le \tilde D \widehat{\bold m}_{W}(B \cap W) \frac{\bold m_{W}(W)}{\bold m_{\widetilde W}(\widetilde W)}.
\]
By the discussion after (H2), we have that \(\bold m_W(W) = |W|\) is bounded by some constant. Moreover, by the definition of Cantor rectangles before (H4), there exists \(c_k>0\) such that, for every \(\widetilde W \) intersecting \(\cantor_k\), \(|\widetilde W| \ge c_k\). Therefore, for some \(C_k >0\),
\[
\widehat{\bold m}_{\widetilde W}( H\cap \brush_{W}(B) \cap \cantor_k \cap \widetilde W) \le C_k \widehat{\bold m}_{W}(B \cap W),
\]
and, by the estimate above,
\[
\begin{split}
  \mu(H)&\le C_k  \widehat{\bold m}_{W}( B \cap W)\int_{\{\widetilde W \in \cM_{\xi^u} \,:\, \widetilde W \cap \underline \cantor_k \neq \emptyset\}} \, d\mu_{\xi^u}(\widetilde W) = C_k \widehat{\bold m}_{W}( B\cap W) \mu(\underline \cantor_k).
\end{split}
\]
By (H4) again, we have \(\mu(\underline \cantor_k) <\infty\). Hence, we conclude the proof by renaming the constant \(C_k\).
\qed

The following lemma is the most important step of the proof.
\begin{lemma}\label{lem:tail-ergodic-atomic}
    For any \(k \in \bN\) there exists \(d \in \bN\) such that \(\fT \cap \Erg_1(\underline \cantor_k) \subseteq \fI_d\).
\end{lemma}
\proof
Fix a Cantor rectangle \(\cantor_k\) and let \(A\) be any positive-measure set in \(\fT \cap \Erg_1( \underline \cantor_k)\). By definition of ergodic hull, \(\Erg_1( \underline \cantor_k) \in \fI_{1}\) and so, by Lemma \ref{lem:K-and-ergodicity}, \(\Erg_1(\underline \cantor_k) \in \fT_{+}\). Therefore, \(A \in \fT_{+}\) as well. Moreover, by the first part of Lemma \ref{lem:ergodic-components} and the fact that \(\cantor_k\) is a positive-measure subset of \(\underline \cantor_k\), we have that \(\Erg_1( \underline \cantor_k) = \Erg_1(\cantor_k)\) mod \(0\). Hence, \(A \subseteq \Erg_1(\cantor_k)\) mod \(0\) and by the definition \eqref{eq:ergodic-hull} of ergodic hull, there exists \(j \in \mathbb Z\) such that \(\cF^{-j}(  \cantor_k) \cap A\) has positive measure. Thus, since \(\cF\) is non-singular, \(\cantor_k \cap \cF^{j}(A)\) has positive measure as well.

We claim that for any \(\ve>0\), there exists \(m \in \bN\) such that 
\begin{equation}\label{eq:final}
    \mu\bigl (\cantor_k \setminus \cF^{m}(A)\bigr ) \le \ve.
\end{equation}
This estimate will be useful later.
To prove the claim, we first recall that according to Lemma \ref{lem:global-manifolds}, for any \(n \in \mathbb Z\), there exists \(A^n = \cF^n(A)\) mod \(0\) such that \(A^n\) is a union of LSMs. Moreover, by the same lemma, the set
\[
\begin{split}
A_{k,j} = \bigl\{x \in \cantor_k \cap \cF^{j}(A) \,:\, W^s(\cF^n (x))\subset A^{n+j}, \ \forall n\in \mathbb Z   \bigr\},
\end{split}
\]
is a full-measure subset of \(\cantor_k \cap \cF^j(A)\). (Notice that \(A_{k,j}\) is a subset of \(\cF^j(A)\) and this explains \(j\) in the superscript). We can pick a point \(x' \in A_{k,j}\) which satisfies Lemma \ref{lem:distortion} for the set \(A_{k,j}\), and whose orbit has infinitely many returns to \(\cantor_k\) (these are typical conditions by (H7) and the aforementioned lemma). Let \(\{n_r\}_{r=1}^{\infty}\) be the sequence of first returns and \(W_r = W^u(\cF^{n_r}(x'))\). By Lemma \ref{lem:distortion}, for any \(\ve >0\) there exists \(r_{*} \in \bN\), such that 
\begin{equation}\label{eq:density-after-returns}
\begin{split}
   \widehat{\bold m}_{W_{r_{*}}}(W_{r_{*}} \setminus \cF^{n_{r_*}}(A_{k,j})) \le \frac{\ve} {C_k},
\end{split}
\end{equation}
where \(C_k >0\) is the constant provided by Lemma \ref{lem:lebw-vsmu}. By the properties of Cantor rectangles and since \(W_{r_*}\) contains the point \(\cF^{n_{r_*}}(x')\in \cantor_k\), we have that \(\cantor_k \subseteq \brush_{W_{r_*}}(W_{r_*})\). Hence,
\begin{equation}\label{eq:cantor-vs-brush}
\begin{split}
\cantor_k \setminus \cF^{n_{r_*} + j}(A) &=  \brush_{W_{r_{*}}} (W_{r_{*}}) \cap \cantor_k \setminus \cF^{n_{r_*} + j}(A) \\
&= (\brush_{W_{r_{*}}} (W_{r_{*}}) \setminus \cF^{n_{r_*} + j}(A) )\cap \cantor_k.
\end{split}
\end{equation}
Moreover, since the sets \(A^{n}\) are equal to \(\cF^n(A)\) modulo null sets,
\begin{equation}\label{eq:brush-vs-tail}
\begin{split}
\brush_{W_{r_{*}}} ( W_{r_{*}}) \setminus  \cF^{n_{r_*} + j}(A) =  \brush_{W_{r_{*}}} ( W_{r_{*}}) \setminus A^{n_{r_*} + j} \quad \text{mod } 0.
\end{split}
\end{equation}
By definition of \(A_{k,j}\), any LSM at \(x \in \cF^{n_{r_*}}(A_{k,j})\) is contained in \(A^{n_{r_*} +j}\). Therefore, recalling that the brush \(\brush_{W_{r_{*}}}(W_{r_{*}})\) is the union of LSMs intersecting \(W_{r_*}\),
\begin{equation}\label{eq:cantor-vs-brush-2}
  \brush_{W_{r_{*}}} ( W_{r_{*}}) \setminus A^{n_{r_*} + j} \subseteq \brush_{W_{r_*}}(W_{r_*} \setminus \cF^{n_{r_*}}(A_{k,j})). 
\end{equation}
Putting together \eqref{eq:cantor-vs-brush}, \eqref{eq:brush-vs-tail} and \eqref{eq:cantor-vs-brush-2}, we have obtained that
\begin{equation}\label{eq:last-cantor-A-tail}
\cantor_k \setminus \cF^{n_{r_*} + j}(A) \subseteq \brush_{W_{r_*}}(W_{r_*} \setminus \cF^{n_{r_*}}(A_{k,j})) \cap \cantor_k \quad \text{mod }0.
\end{equation}
In view of \eqref{eq:last-cantor-A-tail}, we can apply Lemma \ref{lem:lebw-vsmu} with \(H = \cantor_k \setminus \cF^{n_{r_*} + j}(A)\) and \(B= W_{r_*} \setminus \cF^{n_{r_*}}(A_{k,j})\). Therefore,
\[
\begin{split}
\mu\bigl (\cantor_k \setminus \cF^{n_{r_*} + j}(A)\bigr ) \le C_k \widehat{\bold m}_{W_{r_*}}(W_{r_*} \setminus \cF^{n_{r_*}}(A_{k,j}))\le \ve,
\end{split}
\]
where we have used \eqref{eq:density-after-returns} as well. This proves the claim \eqref{eq:final}. 

We now conclude the proof. By recurrence (H7), there exist \(d \in \bN\) such that
\begin{equation}\label{eq:first-return-cantor}
\mu\bigl(\cF^d(\cantor_k) \cap \cantor_k\bigr) >0.
\end{equation}
By \eqref{eq:final}, \eqref{eq:first-return-cantor} and Lemma \ref{lem:auto-int-abstract} of the appendix, there exists \(m \in \bN\) such that
\[
    \mu\bigl(\cF^{d+m}(A) \cap \cF^{m}(A)\bigr) >0,
\]
(to match the notation with Lemma \ref{lem:auto-int-abstract}, \(P = \cantor_k\), \(Q =  \cF^{m}(A)\) and \(T = \cF^d\)). By the equation above and the non-singularity of \(\cF\), we have proved that for any \(A \in \fT \cap \Erg_1(\underline \cantor_k)\) of positive measure,
\[
\mu(\cF^d(A) \cap A) >0
\]
(or, adopting the terminology of Lemma \ref{lem:criterion}, the $\sigma$-algebra \(\fT \cap \Erg_1(\underline \cantor_k)\) is \(d\)-auto-intersecting). Additionally, by Lemma \ref{lem:tail-inv} and the definition of ergodic hull, the \(\sigma\)-algebra \(\fT \cap \Erg_1(\underline \cantor_k)\) is \(\cF^{d}\)-invariant. Therefore, we apply Lemma \ref{lem:criterion} and prove the statement.
\qed
\begin{lemma}\label{lem:final}
\(\fT\) is atomic.
\end{lemma}
\proof
By (H4), we have that \(\cM = \bigcup_{k=1}^{\infty}\cantor_k = \bigcup_{k=1}^{\infty}\Erg_1(\underline \cantor_k)\) mod \(0\). Moreover, by Lemma \ref{lem:K-and-ergodicity}, \(\Erg_1(\underline \cantor_k) \in \fI_1 \subset \fT\). Hence, \(
\fT = \cup_{k=1}^{\infty} \bigl(\fT \cap \Erg_1(\underline \cantor_k)  \bigr)\).
On the other hand, by Lemma \ref{lem:tail-ergodic-atomic}, for each \(k\), there exists \(d_k\) such that \(\fT \cap \Erg_1(\underline \cantor_k) \subseteq \fI_{d_k}\) and, by Lemma \ref{lem:ergodic-components}, each \(\fI_{d_k}\) is atomic. This concludes the proof.
\qed
We are ready to conclude the proof of Theorem \ref{thm:strth}.\\\\
\proofof{Theorem \ref{thm:strth}}
By Lemma \ref{lem:ergodic-components}, \(\fI_1\) is atomic. Denote by \(\{\mathcal E_i\}_{i \ge 1}\) the atoms of \(\fI_1\), a.k.a.\ the \emph{ergodic components} of $\cF$. Since \((\cM, \fA, \mu)\) is \(\sigma\)-finite, we have that there are at most countably many \(\mathcal E_i\). This proves statement a) of Theorem \ref{thm:strth}. 

Let us now consider the map $\mathcal{F}$ restricted to one $\mathcal{E}_{i}$. By Lemma \ref{lem:final}, there exists an atom \(A \in (\mathfrak{T} \cap \cE_i)_{+}\). Moreover, by (H7) there exists a minimum $k_i >0 $ such that $\mu(A\cap \mathcal{F}^{k_i}A)>0$. This implies that $\mathcal{F}^{k_i}A = A$ (mod 0) since $A$ is an atom of \(\fT\) and, by Lemma \ref{lem:tail-inv}, $\fT$ is $\cF^{-1}$-invariant. Denoting $\mathcal{E}_{i}^{j}=\cF^{j}(A)$, \(j \in \{0,...,k_i-1\}\), we have that \(\cup_{j}\mathcal{E}_{i}^{j} = \cE_i\) and, for any \(i,j\),
\begin{equation}\label{eq:inv-tail-sets}
    \cF^{k_i}(\cE_i^j) = \cE_i^j.
\end{equation}
It remains to prove that the map \(\cF^{k_i}\) restricted to \(\mathcal E_{i}^j\) is K-mixing for each \(i\) and \(j\). Let \(\fS_{i,j} = \fS \cap \cE_{i}^j\). For any \(q \in \mathbb Z\), \eqref{eq:inv-tail-sets} implies that
\[
\cF^{q k_i} \fS_{i,j} = (\cF^{q k_i}\fS) \cap \cE_{i}^j,
\]
and the first two properties of Definition \ref{def:K}, for the dynamical system \((\cE_{i}^j, \fA \cap \cE_i^j, \mu, \cF^{k_i})\) and the sub-\(\sigma\)-algebra \(\fS_{i,j}\), follow respectively by Lemmata \ref{lem:first-prop-K} and  \ref{lem:second-K-property}, while the third follows by Lemma \ref{lem:final}. 
\qed

\section{Aperiodic Lorentz gases}\label{sec:aperiodic-LG}

In this section we show that the abstract setting of Section \ref{sec:settingandresults} applies to a certain class of aperiodic Lorentz gases (ALGs) as explained below. 
Let \(\cI \subseteq \bN\) be a infinite set of indices and \(\{\cO_i\}_{i \in \cI}\) a family of \textit{scatterers}, which are \(\cC^3\) closed, pairwise disjoint, convex sets. We consider two possible ambient spaces for ALGs: the two dimensional Euclidean plane and the infinite strip (see Figure~\ref{fig:lorentz-tube}). Accordingly, the billiard table can be one of the following:
\[
\mathcal D = \begin{cases} \bR^2 \setminus \cup_{i \in \cI}\cO_i = \cD_{LG},\\
(\bR \times \mathbb T) \setminus \cup_{i \in \cI}\cO_i = \cD_{LT}. \end{cases}
\]
Here, \(\mathbb T\) is the one-dimensional torus. The subscripts LG stands for Lorentz gas and LT for Lorentz tube, respectively. Of course, we could be more general\footnote{For example, with some extra care the periodic boundary conditions in the Lorentz tube could be replaced by reflecting walls, or we could consider Lorentz tubes made of different types of cells.}, but we consider these two particular cases because \(\cD = \cD_{LG}\) corresponds to the historically first models of ALGs while for \(\cD = \cD_{LT}\) there are available stronger results on recurrence (see Remark \ref{rem:recurrent-billiards} below). 
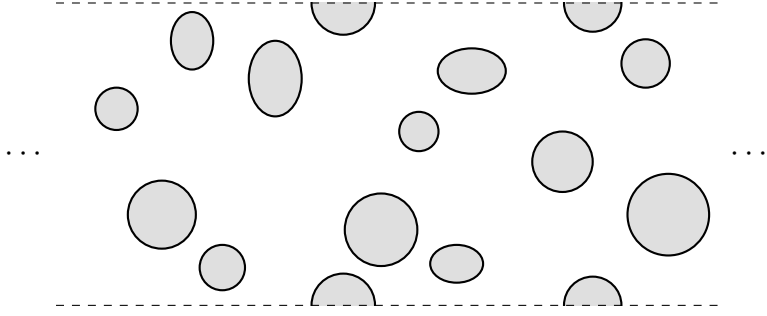
\begin{figure}[ht]
\centering
\begin{tikzpicture}[
    obstacle/.style={fill=gray!25, draw=black, thick}
]
 
\begin{scope}[scale=1]

\begin{scope}
    \clip (-0.5,0) rectangle (8.3,4);
 
    \draw[obstacle] (0.9,1.2) circle (0.45);
    \draw[obstacle] (2.4,3.0) ellipse (0.35 and 0.50);
    \draw[obstacle] (3.8,1.0) circle (0.48);
    \draw[obstacle] (5.0,3.1) ellipse (0.45 and 0.30);
    \draw[obstacle] (6.2,1.9) circle (0.40);
    \draw[obstacle] (7.3,3.2) circle (0.32);
    \draw[obstacle] (1.3,3.5) ellipse (0.28 and 0.38); 
    \draw[obstacle] (1.7,0.5) circle (0.30);
    \draw[obstacle] (4.8,0.55) ellipse (0.35 and 0.25);
 limitato ---
    \draw[obstacle] (0.3,2.6) circle (0.28);
    \draw[obstacle] (7.6,1.2) circle (0.54);
    \draw[obstacle] (4.3,2.3) circle (0.26);

    \draw[obstacle] (3.3,0) circle (0.42);
    \draw[obstacle] (3.3,4) circle (0.42);
 
    \draw[obstacle] (6.6,0) circle (0.38);
    \draw[obstacle] (6.6,4) circle (0.38);
\end{scope}

\draw[dashed] (-0.5,0) -- (8.3,0);
\draw[dashed] (-0.5,4) -- (8.3,4);
 
\node at (-0.9,2) {$\cdots$};
\node at (8.7,2) {$\cdots$};
 \end{scope}
\end{tikzpicture}
\caption{An aperiodic Lorentz tube. The aperiodic Lorentz gas instead extends in both directions in the plane.}
\label{fig:lorentz-tube}
\end{figure}
To define the billiard map \(\cF\), we consider a point-like particle of unit mass and speed that is moving without friction on \(\cD\) and when it hits the boundary of some scatterer, the trajectory reflects according to the law that the angle of incidence equals the angle of reflection. The map \(\cF\) is the Poincaré map corresponding to collisions with the scatterers. The phase space for the billiard dynamics is
\begin{equation}\label{eq:phase-space-billiard}
\cM = \bigcup_{i \in \cI}\cM_{i}, \qquad \cM_{i} = \{(q,v) \in \partial \cO_{i}\times S^1 \,:\, \langle v, n_q \rangle \ge 0\}, 
\end{equation}
where \(n_q\) is the outward normal vector of \(\partial \cO_i\) at \(q\) (\(q\) is the point of collision
and \(v\) is the post-collisional velocity). The standard coordinates on \(\cM\) are \((i, r, \varphi)\), where \(i \in \cI\) identifies the scatterer, \(r\) is the arc-length parameter on \(\partial \cO_i\) and \(\varphi \in [-\pi/2, \pi/2]\) is the angle between \(n_q\) and \(v\) oriented clockwise. It is well known \cite[Lemma 2.35]{MR2229799} that the billiard map preserves a unique (up to factors) absolutely continuous invariant measure \(\mu\), whose density w.r.t.\ Lebesgue is given by 
\begin{equation}\label{eq:invariant-measure-billiard}
d\mu (r, \varphi) = \cos \varphi \, d\varphi dr.
\end{equation}
We are interested in the case where the sum of the arc-length of the scatterers is infinite, so that \(\mu(\cM) = \infty\). For notational convenience, we identify \(\cup_i (\{i\} \times [0, |\partial \cO_i|])\) with \(\mathbb R\), so that we have only one arc-length coordinate and one angle coordinate and we call \(x = (r,\vf)\). We denote by \(\kappa\) the curvature of the boundary, which is the modulus of the second derivative of the arc-length parametrization of \(\partial \cO_i\), and by \(\tau(x)\) the free path of \(x\). Both \(\kappa\) and \(\tau\) are functions \(\cM \to \bR^{+}\). We consider the following class \(\mathcal K\) of ALGs,
\begin{itemize}
    \item[(ALG1)] Any ALG in \(\mathcal K\) is recurrent,
\end{itemize}
\begin{itemize}
\item [(ALG2)] For any ALG in \(\mathcal K\), there exist \(k_m, k_M, \tau_m, \tau_M >0\) such that, for any \(x \in \cM\),
    \[
    \begin{split}
    \tau_m \le \tau(x) \le \tau_M \quad \text{and} \quad k_m \le k(x) \le k_M.
    \end{split}
    \]
\end{itemize}

The main result of this section is the following theorem.

\begin{theorem}\label{thm:K-dec-billiards}
    For any ALG in \(\mathcal K\), either in \(\cD_{LG}\) or \(\cD_{LT}\), the associated billiard map \(\cF: \cM \to \cM\) is K-mixing.
\end{theorem}

Since the literature on billiards is vast and many facts are well-known, we sketch some fairly standard parts of the arguments referring to original or expository references. We start with an extension of \cite[Lemma 4.5]{MR1992667}.
\begin{lemma}\label{lem:K-return-map}
    Fix any \(i \in \cI\) and \(m \in \bN\), and denote by \((\cF^m)_{\cM_i}\) the first return map induced by \(\cF^m\) on \(\cM_{i}\). Then \((\cM_{i}, (\cF^m)_{\cM_i}, \mu)\) is ergodic.
\end{lemma}
\cite[Lemma 4.5]{MR1992667} treats only the case \(m=1\) and \(\cD_{LG}\). However, the proof is a classical Hopf argument and it extends to the present setting with the obvious modifications. To provide some details, ergodicity is proved by constructing \textit{Hopf paths} between almost every pair of points in \(\cM_i\). These Hopf paths consist of sequences of alternating LSMs and LUMs, and their existence between almost every pair of points establishes ergodicity, in analogy with the proof of Lemma~\ref{lem:hopf-on-cantor}. Here the LSMs and LUMs are taken with respect to the map \((\cF^m)_{\cM_i}\). The key observation is that the LSMs and LUMs of \((\cF^m)_{\cM_i}\) coincide with those of \(\cF\), since the latter never intersect the discontinuity set (in the future for LSMs, and in the past for LUMs). The result then follows from the local ergodicity established for billiards and some additional observation on the discontinuity set.

The second result is a standard consequence of Lemma \ref{lem:K-return-map} (see e.g., \cite[Proposition 4.1]{MR4514477} or \cite[Lemma 6.23]{MR2229799}). We report the proof because of the slightly different setting.
\begin{lemma}\label{lem:very-standard}
  Fix any \(m \in \bN\). Then \((\cM, \cF^m,\mu)\) is ergodic.
\end{lemma}
\proof
We want to show that the $\sigma$-algebra \(\fI_m\) given by \(\cF^m\)-invariant sets is trivial. We first prove that \(\fI_m = \fI_1\). Notice that the $\sigma$-algebra \(\fI_m\) is smaller than the one generated by the partition \(\{\cM_{i}\}_{i \in \cI}\) (i.e., elements of \(\fI_m\) are unions of different \(\cM_{i}\)). Indeed, the existence of two distinct \(\cF^m\)-invariant sets \(A\) and \(B\) and \(i_{0} \in \cI\) such that \(\mu(A \cap \cM_{i_{0}}) >0\) and \(\mu(B \cap \cM_{i_{0}}) >0\) contradicts Lemma \ref{lem:K-return-map}. In particular, \(\fI_m\) is atomic. Consider an atom \(P\) of \(\fI_m\) and let \(k >0\) be the smallest positive integer such that \(\mu(P\cap\cF^{k}(P))>0\). We have that \(P\), \(\cF(P)\), \(\cF^2(P)\), ..., \(\cF^{k - 1}(P)\) are disjoint mod 0. Moreover, since \(P\) is an atom and \(P\cap \cF^{k}(P)\) is a positive measure and \(\cF^m\)-invariant subset of \(P\), we have that \(P = \cF^{k}(P)\) mod \(0\). Let \(\cO_0\) be a scatterer whose associated part of the phase space 
\[
\cM_0 =  \{(q,v) \in \partial \cO_{0}\times S^1 \,:\, \langle v, n_q \rangle \ge 0\}
\]
is a subset of \(P\). Take \((q,v) \in \cM_0\), \(\langle v,n_q\rangle \neq 0\), and vary the velocity directed outwardly w.r.t.\ \(\cO_0\). Because of the bounded horizon assumption (ALG2), each trajectory will intersect the boundary of some other scatterer in the point \(q_1 = q_1(q,v)\). Hence, varying clockwise the velocity \(v\), by the convexity of the scatterers, we are bound to find a direction \(v_{*}\) such that \(q_1(q,v_{*})\) results from a tangent collision to some scatterer, which we denote by \(\cO_1\). By convexity of \(\cO_1\), we can also require that \(q_1(q, v) \in \partial \cO_2\) for some other scatterer \(\cO_2 \neq \cO_1\) for every velocity \(v\) that is less clockwise-oriented than \(v_*\) but close enough to \(v_*\). This implies (by continuity of the flow) that there exist non-singular trajectories connecting \(\cO_0\) to \(\cO_1\), \(\cO_1\) to \(\cO_2\) and \(\cO_0\) to \(\cO_2\) (see Figure~\ref{fig:tangent-trajectory}). Denote by \(\cM_1\) and \(\cM_2\) the portions of phase space corresponding to \(\cO_1\) and \(\cO_2\). Since \(\cM_0\) is contained in \(P\), it follows that both \(\cM_1\) and \(\cM_2\) are contained in \(\cF(P)\). However, the existence of non singular trajectories from \(\cM_1\) to \(\cM_2\) shows that \(\cM_2\) is contained in \(\cF^2(P)\) as well. The above fact implies that \(\cF(P)\) and \(\cF^2(P)\) are not disjoint mod \(0\) and this forces \(k\) to be equal to one. Hence, \(P = \cF(P)\) mod \(0\) for each atom \(P \in \fI_m\) so that \(\fI_m = \fI_1\) as anticipated. To conclude, we show that \(\fI_1\) is the trivial $\sigma$-algebra. We have already shown that \(\fI_1\) is smaller than the $\sigma$-algebra generated by \(\{\cM_i\}_{i \in \cI}\). Hence, if the system were not ergodic, the set of scatterers would split into two invariant families. Our geometric assumptions would then easily imply that a non-singular trajectory connects one scatter from one family to a scatterer of the other. Hence, \(\cF\) would transfer a positive-measure subset from an invariant family to the other. This contradiction concludes the proof.
\qed

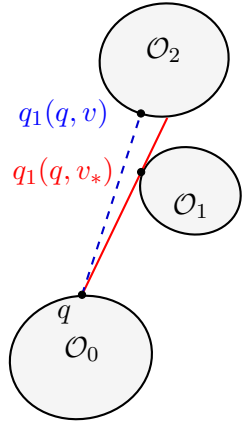
\begin{figure}[ht]
\centering

\begin{tikzpicture}[scale=0.9]

\begin{scope}[rotate around={12:(0,0)}]
  \draw[thick, fill=gray!8] (0,0.08) ellipse (1.05cm and 0.90cm);
\end{scope}
\node[font=\small] at (0,0.2) {$\mathcal{O}_0$};

\begin{scope}[rotate around={-18:(1.6,2.0)}]
  \draw[thick, fill=gray!8] (1.43,2.5) ellipse (0.75cm and 0.63cm);
\end{scope}
\node[font=\small] at (1.6,2.3) {$\mathcal{O}_1$};

\begin{scope}[rotate around={8:(1.2,4.5)}]
  \draw[thick, fill=gray!8] (1.2,4.45) ellipse (0.95cm and 0.83cm);
\end{scope}
\node[font=\small] at (1.2,4.6) {$\mathcal{O}_2$};

\draw[thick, red] (0,1.0) -- (1.253,3.601);

\draw[dashed, blue!80!black, thick] (0,1.0) -- (0.866,3.665);

\filldraw[black] (0,1.0) circle (1.5pt)
  node[below left,font=\small] {$q$};

\filldraw[black] (0.869,2.804) circle (1.5pt)
  node[left=6pt,red,font=\small] {$q_1(q,v_*)$};

\filldraw[black!80!black] (0.866,3.665) circle (1.5pt)
  node[blue, left=8pt,font=\small] {$q_1(q,v)$};

\end{tikzpicture}

\caption{The tangent trajectory with velocity $v_*$ hits
$\partial\mathcal{O}_1$ at $q_1(q,v_*)$. Perturbing this trajectory yields non-singular trajectories connecting $\mathcal{O}_0$ directly to
$\mathcal{O}_2$ at $q_1(q,v)$, as well as connecting $\mathcal{O}_0$ to
$\mathcal{O}_1$ and $\mathcal{O}_1$ to $\mathcal{O}_2$.
}
\label{fig:tangent-trajectory}
\end{figure}

\proofof{Theorem \ref{thm:K-dec-billiards}} We first show that the billiard map associated to any ALG in \(\mathcal K\) satisfies (H1)-(H6), while (H7) is precisely given by (ALG1). These properties have been established for a large class of hyperbolic billiards in compact domains. It is important to notice that, because of the uniformity assumption (ALG2), the particle and its neighboring trajectories would not distinguish between an ALG and a Sinai billiard on the two-dimensional torus. Therefore, given the local nature of most of the arguments, we only sketch part of the proofs, referring to the appropriate literature and highlighting the few differences with the compact case.

\bigskip\noindent
\textbf{(H1)}: The space \(\cM\) in \eqref{eq:phase-space-billiard} is a Riemannian manifold with boundaries (it is topologically the union of infinitely many cylinders) and the measure \(\mu\) of \eqref{eq:invariant-measure-billiard} is equivalent to the standard volume form \(drd\varphi\) on \(\cM\). Setting \((r_n, \varphi_n) = \cF^n(r, \varphi)\), the singularity sets are
\[
\begin{split}
\cS_1 &= \{(r,\varphi) \in \cM \,:\, |\varphi_1| = \pi/2 \quad \text{or} \quad |\varphi| = \pi/2\},\\
\cS_{-1} &= \{(r,\varphi) \in \cM \,:\, |\varphi_{-1}| = \pi/2 \quad \text{or} \quad |\varphi| = \pi/2\}.
\end{split}
\]
These are countably many closed curves.  The \(\cC^3\) regularity of the scatterers implies that the map and its inverse are twice differentiable on \(\cM \setminus \cS_{1}\) and \(\cM \setminus \cS_{-1}\), respectively (see \cite[Theorem 2.33]{MR2229799}). We now review LSUMs and their distortion properties.\\\\
\textbf{(H2)}, \textbf{(H3)}, \textbf{(H6)}: The existence of LSUMs and their absolute continuity is stated in \cite[Theorem 2.3]{MR2237472} for the case \(\cD = \cD_{LG}\) (based on \cite{MR1992667}) and \cite[Theorem 2.1]{MR2741899} for the case \(\cD = \cD_{LT}\). We remark that the proofs mirror those in the compact case, with the sole additional difficulty that the set of singularities for an ALG is in principle much bigger, since it is composed by infinitely many curves. These difficulties are handled in \cite[Lemma 3.2]{MR1992667}, where it was noticed that, although the singularity curves are infinitely many, at every iteration of the map the particle `interacts' with only finitely many of them). As we discuss below, the same arguments prove the measurability of the partitions. 

We outline here a few important details of the construction of LSUMs with the required distortion properties. The standard way to proceed, which altogether yields absolute continuity, measurability and \eqref{eq:equivalent-condiditonal-measure} for the conditional measures, is to introduce \textit{homogeneity strips},
\[
\begin{split}
        \bH_{k} = \bH_k^{+} \cup \bH_k^{-} &= \{(r, \vf) \,:\, \pi/2 -k^{-2} < \vf < \pi/2 -(k+1)^{-2}\}\\
        & \hspace{2cm}\cup \{(r, \vf) \,:\, -\pi/2 + (k+1)^{-2} < \vf < -\pi/2 +k^{-2}\},\\
        \bH_0 &= \{(r,\vf)\,:\, -\pi/2 + k_0^2 < \vf < \pi/2 - k_0^2\}.
    \end{split}
\]
for \(k \ge k_0\) and \(k_0\) big enough. The union \(\mathbb S\) of all the boundaries of these strips define an `extended singularity set' of \(\cF^n\) as follows:
\[
\cS_n^{\bH} = \cS_n \cup \bigcup_{m=0}^n \cF^{-m}(\mathbb S), \qquad \cS_{-n}^{\bH} = \cS_{-n} \cup \bigcup_{m=0}^n \cF^m (\mathbb S).
\]
Recall \eqref{eq:eq-future-sing} for \(\cS_{\pm n}\). Let \(\mathcal Q_{-n}^{\bH}(x)\) be the connected component of \(\cM \setminus \cS_{-n}^{\bH}\) containing \(x\). These are countable approximations of unstable manifolds and \(\cap_{n=1}^{M}\overline{\mathcal Q^{\bH}_{-n}(x)}\) converges as \(M\to \infty\) to the (closure of the) maximal \emph{unstable H-manifold} \(W^u_{\bH}(x)\) passing through \(x\). Stable H-manifolds are defined in a similar way. The partitions into stable and unstable H-manifolds corresponds to the partitions \(\xi^s\) and \(\xi^u\) of the previous section, with \(W_{\bH}^{s/u}(x)\) playing the role of \(W^{s/u}(x)\). For the existence of \(W^u_{\bH}\) and \(W^s_{\bH}\) almost everywhere, the key observation is that the orbit of a.e.\ point does not approach the extended singularity set exponentially fast. We refer to \cite[Theorem 5.17]{MR2229799} for the proof in the compact case and \cite{MR1992667} for ALGs. The first property in \eqref{eq:non-unif-hyp} is a consequence of the fact that
\begin{equation}\label{eq:partition-preimages}
\begin{split}
\cF^{-1}(\overline{W^u_{\bH}(\cF(x))}) = \cF^{-1}\left(\bigcap_{n=1}^{\infty}\overline{\mathcal Q^{\bH}_{-n}(\cF(x))}\right) &= \cF^{-1}\left(\overline{\mathcal Q^{\bH}_{-1}(\cF(x))}\right) \cap \bigcap_{n=1}^{\infty}\overline{\mathcal Q^{\bH}_{-n}(x)} \\
&\subseteq \bigcap_{n=1}^{\infty}\overline{\mathcal Q^{\bH}_{-n}(x)} = \overline{W^u_{\bH}(x)}.
\end{split}
\end{equation}
The second property of \eqref{eq:non-unif-hyp} follows from uniform hyperbolicity of the ALGs in \(\cK\). In particular, \cite[Equation (4.19)]{MR2229799} reads, for \(n \ge 1\) and vectors \(v\) in the unstable cone\footnote{The differential of the billiard map preserves a cone field \cite[Section 4.4]{MR2229799} of \emph{unstable vectors}. These vectors expand uniformly under the derivative cocycle associated to \(\cF\).},
\begin{equation}\label{eq:unstablevectors-grow}
\frac{\|D\cF^n v\|}{\|v\|} \ge \hat c \Lambda^n, \qquad \Lambda = 1 + 2\tau_{m}k_m > 1.
\end{equation}
Since tangent lines to LUMs are aligned with the unstable cone field, \eqref{eq:unstablevectors-grow} and its time reversal prove the second line of \eqref{eq:non-unif-hyp}. For later reference, we also notice that the the constant \(\hat c > 0\) in \eqref{eq:unstablevectors-grow} is uniform\footnote{\(\hat c\) depends on the slopes of the unstable cone field which are uniformly bounded from above and below throughout the phase space.} in \(x \in \cM\). The fact that \(|W_{\bH}^{s/u}(x)|\) is bounded uniformly in \(x\) is a consequence of the fact that the stable and unstable manifolds do not exceed the phase space \(\cM_i\) associated to a single scatterer. These sets are bounded uniformly in the index \(i \in \cI\) because of (ALG2). Finally, (H3) and \eqref{eq:equivalent-condiditonal-measure} are established following \cite[Theorem 5.29]{MR2229799} (see also \cite[Theorem 5.2]{MR2229799}) and (H6) together with \eqref{eq:bound-holonomy} is \cite[Theorem 5.42]{MR2229799}. These results use the fact that the points in \(W_{\bH}^u\) (resp. \(W_{\bH}^s\)) have past (resp. future) trajectories that stay at a controlled distance w.r.t.\ the singularities\footnote{Precisely, by construction, for all \(n\in \bN\), \(\cF^{-n}(W_{\bH}^u)\) is contained in some homoegneity strip.} and are entirely local.\\\\
\textbf{(H5)}: This is established in \cite[Lemma 5.27]{MR2229799}, and a direct analysis of the proof—with no additional considerations beyond those already mentioned—shows that the argument does not distinguish between an ALG and a Sinai billiard with no corners and bounded horizon on the two-dimensional torus.\\\\
\textbf{(H4)}: The fact that a.e.~point \(x \in \cM\) belongs to a Cantor rectangle is established in \cite[Proposition 7.81]{MR2229799} for the compact case. Since the proof is entirely local, it carries over to our setting without modification. However, because this is the first instance in which the condition (H4) is used in the context of ALGs, we briefly sketch the main steps in the proof of \cite[Proposition 7.81]{MR2229799}.
The first technical, yet far-reaching, ingredient is the so-called \textit{Growth Lemma}, established in \cite[Theorem 5.52]{MR2229799}. 

\smallskip\noindent
\textbf{Growth Lemma:} There are constants \(\hat \Lambda >1\), \(\vartheta_1 \in (0,1)\) and \(c_1, c_2 >0\), such that for all \(n\ge 0\) and \(\ve >0\), for any unstable H-manifold \(W_{\bH}^u\),
    \begin{equation}\label{eq:growth-lemma}
         \bold m_{W_{\bH}^u} (r_n (x) \le \ve) \le c_1 (\vartheta_1 \hat \Lambda)^n \bold m_{W_{\bH}^u} (r_0(x) \le \ve/\hat \Lambda^n) + c_2 \ve \bold m_{W_{\bH}^u} (W_{\bH}^u).
        \end{equation}
Here \(r_n(x)\) is the distance between \(\cF^n (x)\) and the endpoints of the unstable H-manifold \(W_n \subseteq \cF^n(W_{\bH}^u)\) such that \(\cF^n(x) \in W_n\). The Growth Lemma gives a bound on how fragmented the forward images of a LUM\footnote{Or, more generally, of any unstable curve, which are curves aligned with the unstable cone field.} may possibly be: the images of a short LUM constitute a family of curves whose average size increases up to a certain threshold. The proof of the Growth Lemma for billiards with no corners is obtained via the so-called one-step expansion (see \cite[Lemma 5.56]{MR2229799}). This expresses the fact that the expansion beats the fragmentation caused by the extended singularity set in just one iteration of the map. The proof of \cite[Lemma 5.56]{MR2229799} applies to ALGs without modifications due to our uniformity assumptions (in particular, note that the ratio \(\tau_M/\tau_m\) is bounded due to (ALG2)). As we will explain shortly, once the Growth Lemma has been established, it can be shown that every unstable H-manifold intersects plenty of stable H-manifolds, with a certain control on their size. In particular, \cite[Theorem 5.67]{MR2229799} asserts the following\footnote{Again, the general statement has any unstable curve in place of \(W^u_{\bH}\).}: for any unstable H-manifold \(W_{\bH}^u \subset \cM\) and \(\bold m_{W^u_{\bH}}\)-almost every \(x \in W_{\bH}^u\), the stable manifold \(W_{\bH}^s(x)\) is non-empty. Furthermore, for some \(C>0\) depending only on the table, and any \(\ve >0\), we have
\begin{equation}\label{eq:long-stable-intersecting-unstable}
\bold m_{W^u_{\bH}} (\{x \in W^u_{\bH} \,:\, r_{\bH}^s (x) < \ve\})\le C\ve.
\end{equation}
Here, \(r_{\bH}^s(x)\) denotes the distance, measured along \(W_{\bH}^s(x)\), between \(x\) and the nearest endpoint of \(W_{\bH}^s(x)\). The proof relies on the Growth Lemma and it is based on the fact that a stable manifold \(W_{\bH}^s(x)\) is short only when the forward trajectory of \(x\) lands close to the enlarged discontinuity set \(\cS_{-1}^{\mathbb H}\) (see \cite[Equation (5.58)]{MR2229799}). This in turn implies that the forward iterate of \(x\) lands close to one of the endpoints of some homegeneous component of \(\cF^n(W)\) \cite[Exercise 5.69]{MR2229799}. Putting everything together, one obtains that, for some \(C>0\), \(\Lambda <1\),
\[
\bold m_{W^u_{\bH}}(\{x \in W^u_{\bH} \,:\, r^s(x) < \ve\}) \le \sum_{n=0}^{\infty} \bold m_{W^u_{\bH}}(\{r_n(x) \le C \Lambda^{-n}\ve\}),
\]
see the equation after \cite[Equation (5.58)]{MR2229799}. This, together with \eqref{eq:growth-lemma}, proves \eqref{eq:long-stable-intersecting-unstable}. The estimate \eqref{eq:long-stable-intersecting-unstable} is sufficient to prove the existence of Cantor rectangles at a.e.\ point. We give some details. We call \textit{solid rectangle} any closed region \(\fO\subset \cM\) bounded by two unstable and two stable H-manifolds. Each such manifold intersects the other two of the opposite type in just one point and constitute the boundary \(\partial \fO\). We say that a stable (resp.\ unstable) H-manifold \(W\) \textit{fully crosses} \(\fO\) if \(W\) intersects \(\fO\) and both the unstable (resp.\ stable) manifolds defining \(\partial \fO\). In correspondence with \(\fO\), there is another set \(\cantor = \cantor (\fO)\), consisting of all those points \(x \in \fO\) such that \textit{both} \(W_{\bH}^s(x)\) and \(W_{\bH}^u(x)\) fully cross \(\fO\) (see Figure~\ref{fig:solid-rectangle}). It is easy to check that any set \(\cantor (\fO)\) corresponding to a solid rectangle \(\fO\) is a Cantor rectangle in the sense of (H4). Since by construction the two stable opposite H-manifolds in \(\partial \fO\) have a positive distance, we have that the lengths of \(W^u_{\bH}(x)\) for \(x \in \cantor(\fO)\) are bounded from below by some constant \(c = c(\cantor(\fO))>0\), in agreement with (H4). To conclude the argument, the estimate \eqref{eq:long-stable-intersecting-unstable} implies the following fact\footnote{To see how, one takes a small piece \(W\) of unstable manifold and applies \eqref{eq:long-stable-intersecting-unstable} to \(\cF^n (W)\). This shows that the long curve \(\cF^n(W)\) posses plenty of small LSMs. Pulling everything back, one obtains a small piece of LUM intersecting many long LSMs.}: for any short enough piece \(W\) of a unstable H-manifold, a proportion, as close to one as desired, of points of \(W\) supports a stable H-manifold longer than, say, \(2|W|\); see \cite[Theorem 5.70]{MR2229799} for the precise statement. We could rephrase this by saying that, at the small scale \(W\), most stable manifolds looks very long (and similarly for unstable manifolds). As noticed in \cite[Proposition 7.81]{MR2229799}, this is enough to show that any point possessing a stable and unstable H-manifold is contained in a sufficiently small Cantor rectangle \(\cantor(\fO)\). Finally, since each H-manifold at \(x\) is contained in \(\cM_{i}\ni x\), we have that each \(\underline \cantor\) is contained in some \(\cM_{i}\), so that \(\mu(\underline \cantor) < \infty\). This concludes the proof of (H4).

Therefore, the main Theorem \ref{thm:strth} applies to each ALG in \(\cK\). By Lemma \ref{lem:very-standard} and Corollary \ref{corlStr}, we conclude the proof of the Theorem.
\qed

\begin{figure}[ht]
\centering

\begin{tikzpicture}

\begin{scope}[rotate=45, scale = 0.8]

\draw[blue!70!black]
  (-0.5,0.3) .. controls (2,0.9) and (4,-0.3) .. (6.5,0.5)
  node[right, font=\small] {$W^{u}_1$};

\draw[blue!70!black]
  (0.5,2.3) .. controls (1.8,2.5) and (2.4,1.85) .. (3.5,2.0)
  .. controls (4.4,2.1) and (4.9,1.6) .. (5.6,1.65)
  node[right, xshift=-3pt, yshift=10pt, font=\small] {$W^{u}_{\bH}(x)$};

\draw[blue!70!black]
  (-0.5,3.7) .. controls (2,3.1) and (4,4.3) .. (6.5,3.5)
  node[right, font=\small] {$W^{u}_2$};

\draw[red!70!black]
  (1,-0.5) .. controls (0.2,1.0) and (1.9,3.0) .. (0.9,4.3)
  node[above, font=\small] {$W^{s}_1$};

\draw[red!70!black]
  (5,-0.5) .. controls (5.6,1.5) and (4.4,2.5) .. (5.2,4.3)
  node[above, font=\small] {$W^{s}_2$};

\draw[red!70!black]
  (3,-0.5) .. controls (2.7,1.2) and (3.4,2.7) .. (3.1,4.3)
  node[above, font=\small] {$W^{s}_{\bH}(x)$};

\filldraw[black] (3.08,2.00) circle (1.5pt)
  node[left = 3pt, font=\small] {$x$};

\end{scope}

\end{tikzpicture}

\caption{A solid rectangle $\fO$ bounded by $W^{u}_{1}, W^{u}_{2}, W^{s}_{1}$ and $W^{s}_{2}$. The point $x$ belongs to the set $\cantor(\fO)$, as its associated H-manifolds $W^{u}_{\bH}(x)$ and $W^{s}_{\bH}(x)$ fully cross $\fO$.}

\label{fig:solid-rectangle}

\end{figure}
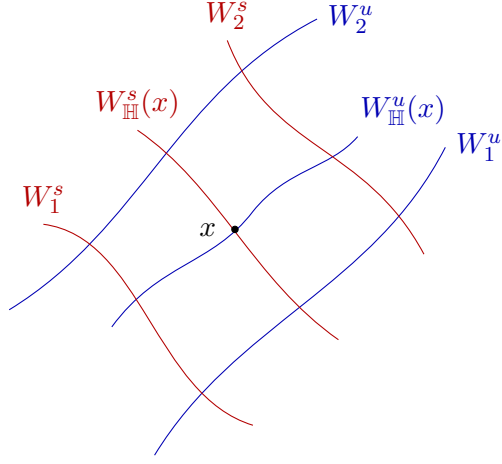

This concludes our work on the main theorem on the ergodic properties of ALGs. As a further application of Theorem \ref{thm:K-dec-billiards} along the line of stochastic properties, we have the following result. For any decreasing positive sequence \(a = \{a_p\}_{p \in \bN}\) such that \(\lim_{p \to \infty}a_p = 0\), we define
\[
\cH_{a} = \{F \,:\, \cM \to \bR \text{ such that } |F(x) - F(y)| \le a_p \text{ for all } |x-y| \le 2^{-p}\}.
\]
We have the following uniform version of \eqref{eq:coalescence}.
\begin{theorem}\label{thm:decay-of-correlations}
For any ALG in \(\cK\), any decreasing positive sequence \(a = \{a_p\}_{p \in \bN}\) and any \(g \in L^1(\mu)\) with \(\int g \,d\mu = 0\), we have
\[
\lim_{n \to \infty} \sup_{\substack{F \in \cH_{a}\\ \|F\|_{\infty} \le 1 }} \int_{\cM} (F \circ \cF^n) g \, d\mu  = 0.
\]      
\end{theorem}
\proof
Recall that, in the case of ALGs, \(\fS\) is the $\sigma$-algebra generated by sets that are mod \(0\) unions of elements of the partition \(\xi^s\) of stable H-manifolds. We also set \(\fS_m = \cF^m \fS\). Fix some sequence \(a = \{a_p\}_{p \in \bN}\) as above. For any \(x\in \cM\), \(F \in \cH_a\) and \(m \in \bN\), set
\begin{equation}\label{eq:def-Fm}
F_m (x) = \int_{\cF^m(W_{\bH}^s(\cF^{-m}(x)))} F(y) \, d\nu_{\cF^m(W_{\bH}^s(\cF^{-m}(x)))}(y).
\end{equation}
Here, \(\nu_{\cF^m(W_{\bH}^s(\cF^{-m}(x)))}\) are the conditional probability measure associated to the measurable partition \(\cF^m \xi^s\). \(F_m\) is \(\fS_m\)-measurable since it is constant on the elements of \(\cF^m \xi^s\). By the uniform hyperbolicity of the ALGs in \(\cK\) (see \eqref{eq:unstablevectors-grow} and the discussion afterward), there exist \(C>0\), \(\Lambda>1\) such that, for any \(m \in \bN\), 
\begin{equation}\label{eq:unif-hyp-sinai}
\sup_{x \in \cM} | \cF^m(W_{\bH}^s(\cF^{-m}(x)))| \le C \Lambda^{-m}.
\end{equation}
Fix \(\ve >0\) and set
\begin{equation}\label{eq:nstar}
 m_{*} = \min \{m \,:\, C \Lambda^{-m}< 2^{-\min \{p \,:\, a_p < \ve\}}\}.
\end{equation}
Notice that \eqref{eq:unif-hyp-sinai} and \eqref{eq:nstar} imply that
\begin{equation}\label{eq:sup-mstar}
\sup_{x \in \cM} | \cF^{m_*}(W_{\bH}^s(\cF^{-m_*}(x)))| \le C \Lambda^{-m_*} < 2^{-\min \{p \,:\, a_p < \ve\}}.
\end{equation}
Using that \(\nu_{\cF^{m_*}(W_{\bH}^s(\cF^{-m_*}(x)))}\) are probability measures, for a.e.\ \(x \in \cM\),
\begin{equation}\label{eq:approx}
\begin{split}
|F_{m_*} (x) - F(x)| &= \biggl|\int_{\cF^{m_*}(W_{\bH}^s(\cF^{-m_*}(x)))} \bigl(F(y) - F(x)\bigr) \, d\nu_{\cF^{m_*}(W_{\bH}^s(\cF^{-m_*}(x)))}(y)\biggr| \\
&\le \int_{\cF^{m_*}(W_{\bH}^s(\cF^{-m_*}(x)))} \bigl|F(y) - F(x)\bigr|\, d\nu_{\cF^{m_*}(W_{\bH}^s(\cF^{-m_*}(x)))}(y) \le \ve.
\end{split}
\end{equation}
In the last estimate, we have used \eqref{eq:sup-mstar} and that \(F \in \cH_a\). We will need \eqref{eq:approx} shortly. Recall that \(L^{\infty}(\fS_{m_*})\) are the \(\fS_{m_*}\)-measurable, essentially bounded functions. By Theorem \ref{thm:K-dec-billiards}, \(\cF\) is K-mixing (w.r.t.\ the $\sigma$-algebra \(\fS\)) and we can apply Lemma \ref{lem:K-1} with \(\fB_m = \fS_m\), \(m=\mu\) and $g$ as in the statement of Theorem \ref{thm:decay-of-correlations}. Thus, there exists \(\overline n\in \bN\) so large that, for each \(n \ge \overline n\),
\begin{equation}\label{eq:K-mixx}
\left|\sup_{\substack{G\in L^{\infty}(\fS_{m_*})\\ \|G\|_{\infty} \le 1}} \int_{\cM} (G \circ \cF^n) g \, d\mu \right| \le \ve.
\end{equation}
Recall that, for each \(F \in \cH_a\), \eqref{eq:def-Fm} defines a function \(F_{m_*} \in L^{\infty}(\fS_{m_*})\). By \eqref{eq:approx} and \eqref{eq:K-mixx} and summing and subtracting \(F_{m_*}\circ \cF^n\), for any \(n \ge \overline n\),
\[
\begin{split}
\left| \sup_{\substack{F \in \cH_{a} \\ \|F\|_{\infty} \le 1}} \int_{\cM} (F \circ \cF^n) g \, d\mu \right| & \le \left|\sup_{\substack{F_{m_*} \in L^{\infty}(\fS_{m_*})\\ \|F_{m_*}\|_{\infty} \le 1} }\int_{\cM} (F_{m_*} \circ \cF^n) g \, d\mu \right| +\!\!\! \sup_{\substack{F \in \cH_{a} \\ \|F\|_{\infty} \le 1}}\!\!\!\|F - F_{m_*}\|_{\infty} \, \|g\|_{L^1} \\[4pt]
&\le \ve + \|g\|_{L^1}\ve.
\end{split}
\]
This concludes the proof of the theorem.
\qed

Observe that the set of all uniformly continuous functions is equal to \(\cup \cH_{a}\) where the union ranges over all possible sequences \(a\) as above. In particular, Theorem \ref{thm:decay-of-correlations} holds for any uniformly continuous observable \(F\), but we have also shown that the decay rate is uniform among any class of observables with a certain regularity (e.g., Lipschitz, Hölder etc.\ with uniform constants). 

We conclude the paper with a remark about the range of applicability of Theorems \ref{thm:K-dec-billiards} and \ref{thm:decay-of-correlations}.
\begin{remark}\label{rem:recurrent-billiards}
Of all the assumption of Theorems \ref{thm:strth} and \ref{thm:K-dec-billiards}, the most difficult to verify is certainly recurrence. Proving recurrence for aperiodic billiards remains a largely open problem. Nevertheless, the results available in the literature provide plenty of applications for Theorem \ref{thm:K-dec-billiards}. In \cite{MR2237472}, it was shown that recurrent ALGs on the plane are typical with respect to a suitable topology. This result was later extended to the case of infinite but `locally finite' horizon configurations \cite{MR2720043}. Furthermore, \cite{MR1992667} established recurrence for \textit{finite modifications} of planar PLGs. These are ALGs obtained by altering the shape and arrangement of a finite number of scatterers of a PLG. The situation is better for billiards for Lorentz gases in strips (a.k.a.\ \emph{Lorentz tubes}). In \cite{MR2741899}, the authors considered an ensemble (i.e., a measured family) of Lorentz tubes with uniform geometric features, together with several generalizations, and proved that a.e.\ realization of the table is recurrent.
\end{remark}

\appendix
\section{Appendix: A few lemmata from measure theory}
Here we collect few general results that are used in Section \ref{sec:proof}. In this section \((X, \fF, m)\) is a general measure space (with \(m(X) < \infty\) or \(m(X) = \infty\)) and \(T\) is a bi-measurable, invertible and non-singular map on \((X, \fF, m)\) (i.e., the measures \(m \circ T^{-1}\) and \(m\) are equivalent). 
\begin{lemma}\label{lem:auto-int-abstract}
For any \(P \in \fF\) such that \(m(T(P) \cap P) >0\) there exists \(\ve >0\) such that for all \(Q\in \fF\) such that \(m (P \setminus Q)\le \ve\) we have
\[
m(T(Q)\cap Q) >0.
\]
\end{lemma}
\proof
For \(n \in \bN\), we set
\[
\ell_n =  \inf\bigg\{ m(T(Q)\cap Q) \,:\, Q \in \fF, \quad m(P\setminus Q) \le \frac 1 n \bigg\},
\]
and we notice that \(\ell_n\) is a non-decreasing non-negative sequence. Call \(\ell = \lim_{n \to \infty} \ell _n \in [0, \infty]\). We first show that \(\ell\) is positive. Assume by contradiction that \(\ell = 0\). In this case \(\ell_n = 0\) for all \(n\) and, by definition of infimum, for every \(n\) there exists \(Q_n \in \fF\) such that
\[
m(P\setminus Q_n) \le \frac 1 n, \qquad m\big(T(Q_n)\cap Q_n\big) \le \ell_n + \frac 1 n = \frac 1 n .
\]
By passing to the subsequence \(n_{k} = 2^k\), we have
\[
\sum_{k=1}^{\infty} m(P\setminus Q_{n_k}) \le \sum_{k=1}^{\infty} 2^{-k} < \infty,
\qquad
\sum_{k=1}^{\infty} m\big(T(Q_{n_k})\cap Q_{n_k}\big) \le\sum_{k=1}^{\infty}  2^{-k}  < \infty.
\]
Hence, by Borel-Cantelli,
\begin{equation}\label{eq:borel-cantelli}
m\left (\bigcap_{m = 1}^{\infty} \bigcup_{k\ge m} P \setminus Q_{n_k}  \right) = 0, \qquad m\left ( \bigcap_{m = 1}^{\infty}\bigcup_{k \ge m} T(Q_{n_k}) \cap Q_{n_{k}}  \right) = 0.
\end{equation}
Since \(T\) is invertible, \(\bigcap_{m = 1}^{\infty} \bigcup_{k\ge m} T(P \setminus Q_{n_k}) = T\big(\bigcap_{m = 1}^{\infty} \bigcup_{k\ge m} P \setminus Q_{n_k} \big)  \) and since \(T\) is non singular, by the first of \eqref{eq:borel-cantelli},
\begin{equation}\label{eq:borel-cantelli-2}
m\left( \bigcap_{m =1}^{\infty} \bigcup_{k\ge m} T(P \setminus Q_{n_k})\right) = m \left( T\bigg(\bigcap_{m = 1}^{\infty} \bigcup_{k\ge m} P \setminus Q_{n_k} \bigg)  \right) = 0.
\end{equation}
We will use \eqref{eq:borel-cantelli} and \eqref{eq:borel-cantelli-2} in a moment. Since \(P \subseteq (P\setminus Q_{n_k}) \cup Q_{n_k}\) and \(T(P) \subseteq T(P\setminus Q_{n_k})) \cup T(Q_{n_k})\), we have, for all \(k \ge 1\),
\[
T(P)\cap P  \subseteq (T(Q_{n_k}) \cap Q_{n_{k}}) \cup (P\setminus Q_{n_k}) \cup T(P\setminus Q_{n_k}).
\]
Consequently, for any \(m\ge 1\),
\[
T(P)\cap P \subseteq \bigcup_{k \ge m} (T(Q_{n_k}) \cap Q_{n_{k}} )\cup (P\setminus Q_{n_k}) \cup T(P\setminus Q_{n_k}),
\]
and intersecting over all \(m\),
\[
\begin{split}
T(P)&\cap P  \subseteq \bigcap_{m=1} ^{\infty}\bigcup_{k \ge m} (T(Q_{n_k}) \cap Q_{n_{k}}) \cup (P\setminus Q_{n_k}) \cup T(P\setminus Q_{n_k})\\
& = \left( \bigcap_{m=1} ^{\infty}\bigcup_{k \ge m} T(Q_{n_k}) \cap Q_{n_{k}}\right) \cup \left( \bigcap_{m=1} ^{\infty}\bigcup_{k \ge m} (P\setminus Q_{n_k}) \right) \cup \left( \bigcap_{m=1} ^{\infty}\bigcup_{k \ge m} T(P\setminus Q_{n_k})\right).
\end{split}
\]
The equation above, \eqref{eq:borel-cantelli} and \eqref{eq:borel-cantelli-2} imply that \(m(T(P)\cap P) = 0\), which is a contradiction. This proves that \(\ell > 0\) and so \(\ell_n >0\) for some \(n\). The last assertion implies the statement with \(\ve = 1/n\).
\qed

Recall that a $\sigma$-algebra \(\mathfrak H \subseteq \fF\) is \(T^d\)-invariant if \(T^{-d}(A) \in \mathfrak H\) for any \(A \in \mathfrak H\). The next lemma is in the spirit of \cite[Lemma 2.1]{MR3009111} and \cite[Proposition A.2]{MR3433572}, both based on the results in \cite{MR143873}.
\begin{lemma}\label{lem:criterion}
Let \(\mathfrak H \subseteq \fF\) be a \(T^{d}\)-invariant \(\sigma\)-algebra  and assume that, for any \(P \in \mathfrak H\) with \(m(P) >0\) we have \(m(T^{d}(P) \cap P)>0\) (we say that \(\mathfrak H\) is \(d\)-auto-intersecting). Then \(\mathfrak H \subseteq \fI_{d}\). 
\end{lemma}
\proof
Let \(P \in \mathfrak H\). By the invariance of \(\mathfrak H\) both \(P_1 = T^{-d}(P) \setminus P\) and \(P_2 = P \setminus T^{-d}(P)\) belong to \(\mathfrak H\). By construction, for \(i \in \{1,2\}\), \(T^{-d}(P_i) \cap P_i = \emptyset\) and, since \(T\) is invertible, we have
\[
   T^d(P_i) \cap P_i = \emptyset.
\]
Since \(\mathfrak H\) is \(d\)-auto-intersecting, the equation above implies that \(m(P_1)=m(P_2) = 0\). Hence, noticing that \(P_1 \cup P_2 = P \triangle T^{-d}(P)\), we have \(m \bigl (P \triangle  T^{-d}(P)\bigr) = 0\). This proves the claim.
\qed

We now discuss the \(K\)-property. Recall that \((X,\fF, m, T)\) is K-mixing w.r.t.\ the $\sigma$-algebra \(\fB \subset \fF\) if
\begin{equation}\label{eq:def_k_again}
\mathfrak{B} \subseteq T \mathfrak{B}, \qquad \bigvee_{n = 0}^{\infty} T^{n}\mathfrak{B} = \fF, \qquad \bigcap_{n = 0}^{\infty} T^{-n}\mathfrak{B} =\mathfrak N.
\end{equation}
For \(p \in \mathbb Z\), we set \(\fB_p = T^p \fB\). The following is a characterization of K-mixing in terms of correlations based on the abstract results of \cite{Lin} and similar to \cite[Theorem 1.3.3]{MR1450400} where exactness was considered.
\begin{lemma}\label{lem:K-1}
    Let \(\fB\) a \(\sigma\)-algebra satisfying the first two properties of \eqref{eq:def_k_again}. \((X,\fF, m, T)\) is K-mixing w.r.t.\ \(\fB\) if and only if for any \(p \in \mathbb Z\), \(g\in L^{1}(m)\), \(\int_X g \, dm = 0\),
    \[
    \lim_{n \to \infty}\sup_{\substack{F\in L^{\infty}(\fB_p)\\
    \|F\|_{\infty} \le 1}} \int_{X} (F\circ T^n) g \, dm = 0.
    \]
\end{lemma}
\proof Let us first assume that \((X,\fF, m, T)\) is K-mixing. Fix \(p \in \mathbb Z\). By the first equation of \eqref{eq:def_k_again}, for any \(F \in L^{\infty}(\fB_p)\) and \(n \in \bN\), we have that \(F\circ T^n\) is \(\fB_p\)-measurable. Hence, denoting by \(\mathbb E (\cdot|\cdot)\) the conditional expectation,
\[
\sup_{\substack{F\in L^{\infty}(\fB_p)\\
    \|F\|_{\infty} \le 1}} \int_{X} (F\circ T^n) g \, dm = \sup_{\substack{F\in L^{\infty}(\fB_p)\\
    \|F\|_{\infty} \le 1}} \int_{X} (F\circ T^n) \mathbb E(g| \fB_p) \, dm.
\]
Since \(\int_X \mathbb E(g| \fB_p) \, dm = \int_X g \, dm = 0 \) and, by the third equation of \eqref{eq:def_k_again}, \(\cap_{n}\fB_{p-n} = \mathfrak N\), we can apply \cite[Corollary 4.1]{Lin} with \(\Sigma = \fB_p\) and we prove the first implication. For the other direction, observe that if \((X,\fF, m, T)\) is not K w.r.t.\ \(\fB\) there exists a set \(A \in \cap_{n \ge 0} \fB_{-n}\) with \(m(A), m(A^c) >0\) and a function \(g_A \in L^1(m)\) with \(\int_X g_A dm = 0\) and \(\int_{A} g_A dm >0\). Moreover, by definition of \(A\), for each \(n\in \bN\), there exists \(A_n \in \fB\) such that \(A= T^{-n}(A_n)\). Then, for any \(n\in \bN\),
\[
\int_{X} (\bold 1_{A_n} \circ T^n) g_A \, dm = \int_A g_A \, dm >0.
\]
This contradicts the formula in the statement of the Lemma in the case \(p =0\) and concludes the proof.
\qed
In particular, by an approximation argument and Lemma \ref{lem:K-1}, one can obtain the following result, already noticed in \cite[Theorem 3.5]{MR3235352}.
\begin{corollary}\label{cor:decay-K-mix}
If \((X,\fF, m, T)\) is K-mixing, for any \(F: X \to \bR\) which belongs to the \(L^{\infty}(m)\)-closure of \(\cup_{p\ge 0}L^{\infty}(\fB_p)\) and \(g\in L^1(m)\), \(\int_X g \, dm = 0\), we have
\[
\lim_{n \to \infty} \int_{X} (F\circ T^n) g \, dm = 0.
\]
\end{corollary}

\bibliographystyle{alpha}
\bibliography{modernity2.bib}

\end{document}